\documentclass[reqno]{amsart}

\usepackage{external/takodachi}

\newcommand{\Gauss}{(\mathsf{G})}

\newcommand{\Bogomolnyi}{(\mathsf{B})}

\title
{
	Convergence of the self-dual abelian Higgs gradient flow
} 
\author{Jason Zhao}
\address{Department of Mathematics, University of California, Berkeley, 94720}
\email{zhao.j@berkeley.edu}
\date{\today}

\begin{document}

\begin{abstract}
	Given an initial data configuration $(A^{\mathrm{in}}, \phi^{\mathrm{in}})$ on $\mathbb R^2$ such that the self-dual abelian Higgs energy is near the minimum energy within its topological class, we prove that its evolution under the self-dual abelian Higgs gradient flow in temporal gauge converges exponentially as $t \to \infty$ with respect to the $(H^1 \times L^2)$-metric to a minimiser of the energy. Furthermore, we show that the convergence of the scalar field $\phi$ may be upgraded to the $H^1$-metric provided the additional assumption on the potential that $A^{\mathrm{in}} \in L^p (\mathbb R^2)$ for $2 < p < \infty$. As a corollary, we obtain a quantitative stability for the self-dual abelian Higgs energy which improves upon the previous result of Halavati \cite{Halavati2024} and partially resolves the open problem posed in his article. 
\end{abstract}
\maketitle

\tableofcontents

\section{Introduction}

The \textit{abelian Higgs energy}, also known as the Ginzburg-Landau energy, is a gauge-invariant functional on configurations $(A, \phi)$ consisting of a complex scalar field $\phi : \R^2 \to \C$ and a real-valued connection $1$-form $A = A_j \, dx^j$ on $\R^2$. Denoting $\bfD_\mu = \partial_\mu - i A_\mu$ the covariant derivative associated with the connection, and $F_{\mu\nu} = \partial_\mu A_\nu - \partial_\nu A_\mu$ for the corresponding curvature $2$-form, the energy is given by 
	\[
		\cE [A, \phi] 
			:= \tfrac12 \int_{\R^2} \Bigl( |\bfD_1 \phi|^2 + |\bfD_2 \phi|^2 + |F_{12}|^2 + \tfrac\lambda4 (1 - |\phi|^2)^2 \Bigr) \, dx,
	\]
where $\lambda > 0$ is a coupling constant. Its critical points, termed \textit{vortices}, are examples of topological defects; for a survey of this subject, see the monograph of Manton-Sutcliffe \cite{MantonSutcliffe2004}. For configurations with finite abelian Higgs energy, the \textit{topological degree},
	\[
		\deg[A, \phi] 
			:= \tfrac{1}{2\pi} \int_{\R^2} \epsilon_{jk} \partial^j \Bigl( A^k + \Im(\overline \phi \bfD^k \phi) \Bigr) \, dx,
	\]
is well-defined and integer-valued; see Section \ref{sec:topology}. In the \textit{self-dual} setting, where the coupling constant is taken as $\lambda = 1$, Bogomoln'yi \cite{Bogomolny1976} formally derived a lower bound on the energy of a configuration in terms of its topological degree, 
	\[
		\cE[A, \phi] 
			\geq \pi \bigl|\deg [A,\phi]\bigr|.
	\]
Jaffe-Taubes \cite[Chapter III]{JaffeTaubes1980} showed that this bound is saturated by the vortices, and, for each $N \in \Z$, constructed the moduli space $\cM^N$, the set of critical points for the self-dual abelian Higgs energy with topological degree $N$. It is of interest to study the stability of these vortices -- in this article, we consider the dynamic stability under the associated gradient flow. 

Given a time-dependent scalar-field $\phi : [0, T] \times \R^2 \to \C$ and a connection $1$-form $A = A_t \, dt + A_j \, dx^j$ on $[0, T] \times \R^2$, the \textit{self-dual abelian Higgs gradient flow}, also known as the {Gorkov-\'Eliashberg} equations or the {time-dependent Ginzburg-Landau} equations \cite{GorkovEliashberg1968}, is the system of equations
	\begin{equation}\label{eq:AHG}\tag{AHG}
		\begin{split}
			\bfD_t \phi - \bfD_1 \bfD_1\phi - \bfD_2 \bfD_2 \phi
				&= \tfrac12(1 - |\phi|^2) \phi, \\
			F_{t1}
				&= \Im(\overline \phi \bfD_1 \phi) - \partial_2 F_{12},\\
			F_{t2} 
				&= \Im(\overline \phi \bfD_2 \phi) + \partial_1 F_{12}.
		\end{split}
	\end{equation}
The first equation is a covariant non-linear heat equation for the scalar field, and the second and third equations are \textit{Ampere's law}, relating the magnetic flux to the charge current. This system is covariant in that it is invariant under the gauge transformation
	\[
		(A, \phi) 
			\mapsto (A + d\chi, e^{i\chi} \phi), 
	\]
for any real scalar field $\chi : [0, T] \times \R^2 \to \R$. It is also natural to formally impose the boundary conditions at spatial infinity,
    \begin{equation*}
    \begin{split}
        A 
            &\overset{|x| \to \infty}{\longrightarrow} 0,\\
        |\phi| 
            &\overset{|x| \to \infty}{\longrightarrow} 1 .\\
    \end{split}
    \end{equation*}
The latter is often known as a topological boundary condition, and is imposed by requiring finite integral of the $\phi^4$-potential $V(\phi) := \tfrac14 (1 - |\phi|^2)^2$. The former is a matter of technical convenience, as it rules out ``non-physical" representatives within each gauge-equivalence class; for this, we say $(A, \phi)$ is \textit{admissible} if $A \in L^p (\R^2)$ for some $2 < p < \infty$. In concert, we say $(A, \phi) \in H^1_\loc (\R^2)$ is a \textit{finite-energy configuration} if it is an admissible configuration with finite abelian Higgs energy.

For sufficiently regular configurations $(A, \phi)$ solving \eqref{AHG}, the following energy identity holds, 
	\[
		\frac{d}{dt} \cE[A, \phi] 
			= - \int_{\R^2} | \bfD_t \phi|^2 + |F_{t1}|^2 + |F_{t2}|^2 \, dx.
	\]
In fact, by fixing the \textit{temporal gauge},
	\begin{equation}\label{eq:temporal-gauge}\tag{T}
		A_t 
			= 0,
	\end{equation}
the system \eqref{AHG} arises as the gradient flow of the self-dual abelian Higgs energy. It is therefore natural, from \texttt{(i)} an evolutionary PDE perspective, and \texttt{(ii)} as a mechanism for regularising generic finite-energy configurations towards ground-states, to study the Cauchy problem for \eqref{AHG} under the temporal gauge \eqref{temporal-gauge} with finite-energy initial data,
	\[
		(A, \phi)_{|t = 0} = (A^{\mathrm{in}}, \phi^{\mathrm{in}}) , \qquad \cE[A^{\mathrm{in}}, \phi^{\mathrm{in}}] < \infty.
	\]
Our main result states that, for initial data which nearly saturates the Bogomoln'yi bound, the ensuing flow converges exponentially to the moduli space of vortices,

\begin{theorem}[Convergence of self-dual flow to $\cM^N$]\label{thm:main}
	Let $N \in \Z$ be an integer, and suppose $\epsilon_* \ll_{|N|} 1$ and $\gamma \ll_{|N|} 1$ are sufficiently small, and $C \gg_{|N|} 1$ is sufficiently large. Consider a finite-energy configuration $(A^{\mathrm{in}}, \phi^{\mathrm{in}})$ on $\R^2$ with topological degree $N$ and self-dual abelian Higgs energy satisfying
		\[
			\cE [A^{\mathrm{in}}, \phi^{\mathrm{in}}] - \pi |N| < \epsilon_*,
		\]
	and let $(A, \phi)$ be the solution to the self-dual abelian Higgs gradient flow \eqref{AHG} on $[0, \infty) \times \R^2$ in temporal gauge \eqref{temporal-gauge} with initial data $(A, \phi)_{|t = 0} = (A^{\mathrm{in}}, \phi^{\mathrm{in}})$. Then there exists a unique self-dual vortex $(A^\infty, \phi^\infty) \in \cM^N$ with topological degree $N$ such that 
		\begin{equation}\label{eq:main-converge}
			\bigl\| A(t) - A^\infty\bigr\|_{H^1}^2 + \bigl\| \phi(t) - \phi^\infty\bigr\|_{L^2}^2
				\leq C  e^{-\gamma t} \Bigl( \cE [A^{\mathrm{in}}, \phi^{\mathrm{in}}] - \pi |N|  \Bigr) \qquad \text{for all $t \geq 0$.}
		\end{equation}
	Furthermore, suppose that the magnetic potential satisfies $A^{\mathrm{in}} \in L^p (\R^2)$ for $2 < p < \infty$, then 
		\begin{equation}\label{eq:main-converge-2}
			\bigl\| \nabla \phi(t) - \nabla \phi^\infty\bigr\|_{L^2}^2
				\leq C \bigl\| A^{\mathrm{in}} \bigr\|_{L^p}^2 \,  e^{-\gamma t} \Bigl( \cE [A^{\mathrm{in}}, \phi^{\mathrm{in}}] - \pi |N|  \Bigr) \qquad \text{for all $t \geq 0$.}
		\end{equation}
\end{theorem}

\begin{remark}
	The temporal gauge \eqref{temporal-gauge} fixes the gauge freedom up to time-independent gauge transformations $\chi \equiv \chi(x)$. This gauge is natural in that it renders the flow map $t \mapsto (A(t), \phi(t))$ formally $L^2$-orthogonal to the gauge orbits \cite[Section 2.8]{MantonSutcliffe2004}. It is not difficult to see that the bound \eqref{main-converge} is invariant with respect to this residual gauge freedom, while \eqref{main-converge-2} is not. Nonetheless, we can pass $H^1$-difference bounds for a scalar field $\phi$ to the gauge-transformed field $e^{i \chi} \phi$ by interpolation, provided that the transformation has regularity $\nabla \chi \in L^p (\R^2)$ for some $2 < p \leq \infty$, see Lemma \ref{lem:transform}. 
\end{remark}

\begin{remark}
	In Section \ref{sec:H1-phi}, where we prove \eqref{main-converge-2}, we will see that the $L^\infty_t L^p_x$-norm of the magnetic potential is $O_N (1)$ if the initial data satisfies the \textit{Coulomb gauge},
		\begin{equation}\label{eq:coulomb} \tag{C}
			\partial^\ell A_\ell = 0.
		\end{equation}
	It is here where we use the notion of admissibility, as it rules out representatives with non-trivial spatial asymptotics, e.g. $A = d \chi$ is gauge-equivalent to the trivial connection $A \equiv 0$, and it satisfies the Coulomb gauge whenever $\chi$ is harmonic. We only use $A \in L^p (\R^2)$ qualitatively, namely to ensure that the Biot-Savart formula holds, and it could probably be relaxed with more careful distribution theory. 
\end{remark}

\begin{remark}
	There are no non-minimising critical points of the self-dual abelian Higgs energy \cite[Theorem 1.1]{JaffeTaubes1980}, so, from the point of view of the energy identity, there does not seem to be any obstructions to the energy dissipating towards its minimum starting from generic data. It is conceivable then that one could extend Theorem \ref{thm:main} to arbitrary initial configurations after developing suitable compactness theory.
\end{remark}

As an immediate corollary, we obtain a quantitative stability for the self-dual energy, stating that if the energy of a configuration is near-minimising, then it must be close with respect to the $(H^1 \times L^2)$-metric to a vortex. This improves the result of \cite{Halavati2024}, and partially answers the problem posed in his article concerning stability with respect to the $(H^1 \times H^1)$-metric, 

\begin{corollary}[Stability of the self-dual energy]\label{cor:stability}
	Let $N \in \Z$ be an integer, and suppose $\epsilon_* \ll_{|N|} 1$ is sufficiently small and $C \gg_{|N|} 1$ is sufficiently large. Consider a finite-energy configuration $(A^{\mathrm{in}}, \phi^{\mathrm{in}})$ on $\R^2$ with topological degree $N$ and self-dual abelian Higgs energy satisfying
		\[
			\cE [A^{\mathrm{in}}, \phi^{\mathrm{in}}] - \pi |N| < \epsilon_*,
		\]
	Then  
		\begin{equation}\label{eq:main-stable}
			\inf_{(A^\infty, \phi^\infty) \in \cM^N} \bigl\| A^{\mathrm{in}} - A^\infty \bigr\|_{H^1}^2 + \bigl\| \phi^{\mathrm{in}} - \phi^\infty \bigr\|_{L^2}^2 \leq C \Bigl( \cE [A^{\mathrm{in}}, \phi^{\mathrm{in}}] - \pi |N| \Bigr). 
		\end{equation}
	Furthermore, if $A^{\mathrm{in}} \in L^p (\R^2)$ for $2 < p < \infty$, then 
		\begin{equation}\label{eq:main-stable-2}
			\inf_{(A^\infty, \phi^\infty) \in \cM^N} \bigl\| A^{\mathrm{in}} - A^\infty \bigr\|_{H^1}^2 + \bigl\| \phi^{\mathrm{in}} - \phi^\infty \bigr\|_{H^1}^2 \leq C \| A^{\mathrm{in}} \|_{L^p}^2 \Bigl( \cE [A^{\mathrm{in}}, \phi^{\mathrm{in}}] - \pi |N| \Bigr). 
		\end{equation}
\end{corollary}

\begin{remark}
	Halavati \cite{Halavati2024} asked if there exists a sequence of finite-energy configurations $\{(A^{(n)}, \phi^{(n)})\}_n$ with topological degree $N$ such that 
		\[
			\lim_{n \to \infty} \cE[A^{(n)}, \phi^{(n)}] = \pi |N| \qquad \text{ and } \qquad \lim_{n \to \infty} \frac{\cE[A^{(n)}, \phi^{(n)}] - \pi |N|}{\inf_{(A^\infty, \phi^\infty) \in \cM^N} \| \nabla \phi^{(n)} - \nabla \phi^\infty \|_{L^2}^2} = 0.
		\]
	By \eqref{main-stable-2}, this cannot hold if $\{ A^{(n)}\}_n$ is uniformly bounded in $L^p (\R^2)$ for some $2 < p < \infty$. It seems to be of interest to see whether one can remove this restriction, though, as pointed out in  the remarks following Theorem \ref{thm:main}, this problem is quite sensitive to the choice of gauge. 
\end{remark}

\subsection{Previous results}

The long-time dynamics of the abelian Higgs gradient flow has a fairly long, albeit sparse, history. We summarise a few of the known results, and refer the reader to the relevant chapter of Manton-Sutcliffe \cite[Chapter 7.7.2]{MantonSutcliffe2004} for a fuller survey of the subject. 

The earliest rigorous work was done by Demoulini-Stuart \cite{DemouliniStuart1997} (see also the survey article of \cite{Stuart1996}) on the self-dual abelian Higgs gradient flow \eqref{AHG} on $\R^2$ in temporal gauge \eqref{temporal-gauge} for a class of smooth, localised initial data satisfying certain pointwise constraints. In this class, which contains configurations which are not necessarily near energy-minimising, they prove $H^2$-convergence of the flow. They exploit the self-duality via Pohozaev-type identities to propagate the localisation uniformly-in-time, which allows them to upgrade sequential weak-convergence of the flow allotted by the energy identity to sequential strong-convergence. They can then reduce to the low-energy regime, wherein they conclude their main result using the modulation theory developed by Stuart \cite{Stuart1994} for the hyperbolic abelian Higgs equation. 

For arbitrary coupling constant $\lambda > 0$, there exist unique finite-energy critical points $(A^{\lambda, N}, \phi^{\lambda, N})$ of the abelian Higgs energy which are $N$-equivariant, that is, taking the form,
	\[
		\phi^{\lambda, N} \equiv \phi^{\lambda, N} (r) e^{i N \theta} , \qquad A \equiv A_\theta^{\lambda, N} (r) \, d \theta,
	\]
for radial profiles $\phi^{\lambda, N}(r)$ and $A_\theta^{\lambda, N}(r)$. Gustafson \cite{Gustafson2002} considered the $H^1$-orbital stability of these vortices for the corresponding gradient flows (along with the hyperbolic abelian Higgs equation) in temporal gauge, proving stability in the cases when $\lambda < 1$ for all $N \in \Z$, and $\lambda \geq 1$ for the $1$-equivariant vortices $N = \pm 1$, and instability for $\lambda > 1$ and $|N| > 1$. The proof is based on modulation theory and stability of the linearised operator -- see \cite{GustafsonSigal2000} and the references therein. 

The ideas in this article, particularly those in Section \ref{sec:decay}, most closely resemble those in Demoulini \cite{Demoulini2013a} and Hassell \cite{Hassell1993}. The former showed $L^2$-convergence of the self-dual abelian Higgs flow on a closed Riemannian surface and in Coulomb gauge, i.e. $\partial^\ell A_\ell = 0$ for all $t \geq 0$, while the latter considered $C^k$-convergence for the Yang-Mills-Higgs gradient flow on $\R^3$ in temporal gauge. The common theme between these works and ours is in proving covariant decay estimates for the \textit{Bogomoln'yi tension field}. The self-duality manifests in the miraculous fact that these covariant quantities satisfy a system of massive parabolic equations with quadratic non-linearity, which is thereby amenable to perturbative arguments. 

As mentioned earlier, our stability result, Corollary \ref{cor:stability}, has precedent in the work of Halavati \cite{Halavati2024}, who showed that 
	\[
			\inf_{(A^\infty, \phi^\infty) \in \cM^N} \bigl\| F_{12}^{\mathrm{in}} - F_{12}^\infty \bigr\|_{L^2}^2 + \bigl\| \phi^{\mathrm{in}} - \phi^\infty \bigr\|_{L^2}^2 \leq C \Bigl( \cE [A^{\mathrm{in}}, \phi^{\mathrm{in}}] - \pi |N| \Bigr),
	\]
for near-minimising configurations. To compare, here we recover closeness of the full derivative of the magnetic potential rather than just the curl, and prove closeness of the derivative of the scalar field under suitable assumptions. In particular, \eqref{main-stable-2} partially resolves the open problem posed in the article. 

\subsection{Outline of the article}

By the reflection symmetry $(x^1, x^2) \mapsto (x^1, -x^2)$, we consider without loss of generality non-negative topological degree $N \in \N_0$. A suitable global well-posedness theory for \eqref{AHG} with finite-energy initial data is developed in Appendix \ref{sec:gwp}, so the concerned reader should banish any qualms regarding qualitative aspects of the initial data problem. 

After reviewing some notation and geometry in Section \ref{sec:prelim}, we begin our analysis in Section \ref{sec:topology} by defining the topological degree of a finite-energy configuration, and identifying a sufficient condition for which the degree is preserved along a $1$-parameter family of configurations. In Section \ref{sec:self-dual}, we reveal the self-dual structure of both the abelian Higgs energy and its gradient flow, and recall the definition of the moduli space of self-dual vortices $\cM^N$. The technical groundwork is contained in Sections \ref{sec:smooth}, where we prove smoothing estimates on unit time-scales, and Section \ref{sec:decay}, where we identify a gauge-covariant quantity known as the Bogomoln'yi tension field which decays exponentially. Recasting the flow in terms of the Bogomoln'yi tension field, we conclude Theorem \ref{thm:main} in Sections \ref{sec:H1-L2}-\ref{sec:H1-phi} by interpolating between the exponential decay estimates and unit-time scale bounds.

\subsubsection*{Acknowledgements}

The author would like to thank Sung-Jin Oh for encouraging discussions in the course of conducting this work, and for making useful suggestions on the preliminary manuscripts. The author also thanks Sky Cao for pointing out a few key typos in an earlier version of this article. The author was partially supported by the National Science Foundation CAREER Grant under \texttt{NSF-DMS-1945615}.

\section{Preliminaries}\label{sec:prelim}

\subsection{Notation}

\subsubsection*{Indices}

We employ the usual Einstein summation convention of summing repeated upper and lower indices. The Levi-Civita symbol $\epsilon_{jk}$ is the anti-symmetric tensor
    \[
        \epsilon_{jk} 
            := 
            \begin{cases}
                +1 & \text{if $(j, k) = (1, 2)$}, \\
                -1 & \text{if $(j, k) = (2, 1)$},\\
                0
                    & \text{otherwise.}
            \end{cases}
    \]

\subsubsection*{Derivatives}

We will use $\bfD^{(n)}$ and $\nabla^{(n)}$ for $n$-fold iterated spatial derivatives. More precisely, in equations, this will be a placeholder for any $n$-fold combination of derivatives, e.g.
    \[
        \bfD^{(n)} = \bfD_{j_1} \cdots \bfD_{j_n} 
            \qquad 
            \text{for some indices $j_1, \dots, j_n \in \{1, 2\}$}
    \] 
while the norms of these expressions will represent the sum total of all such combinations, e.g. 
    \[
        \bigl\| \bfD^{(n)} \phi \bigr\|_{L^p_x} 
            := \sum_{j_1, \dots, j_n \in \{1, 2\} } \bigl\| \bfD_{j_1} \cdots \bfD_{j_n} \phi \bigr\|_{L^p_x}.
    \]

\subsubsection*{Multi-linear forms}

We will use dots as catch-all notation for multi-linear expressions which are gauge-covariant or gauge-invariant, depending on the context. For example, 
    \[
        \Phi \cdot \Psi \cdot \Gamma 
            = a \overline \Phi \Psi \Gamma + b\Phi \overline\Psi \Gamma + c\Phi \Psi \overline\Gamma , \qquad \text{for some $a, b, c \in \C$}
    \]
The reader should simply keep in mind that these expressions behave well with respect to the product rule and its covariant analogues. 

\subsubsection*{Asymptotic notation}

The notations $X \lesssim Y$ denote the inequality $X \leq C Y$ for an implicit constant $C > 0$. Moreover, $X \sim Y$ is an abbreviation for $X \lesssim Y \lesssim X$. We will write $X \ll Y$ for $CX \leq Y$ and conversely $X \gg Y$ for $X \geq C Y$ for a sufficiently large constant $C > 0$. We indicate the dependence of implicit constants on parameters by subscripts, e.g. $X \lesssim_s Y$ denotes $X \lesssim C(s) Y$. 

\subsection{Covariant identities and inequalities}

It will be useful to record some standard algebraic identities and pointwise inequalities for covariant derivatives. Throughout, $A = A_\bfa \, dx^\bfa$ will be a connection $1$-form on a (subset of) Euclidean space, e.g. $\R^2$, $[0, T] \times \R^2$, and $\Phi$ will be a complex scalar field with the appropriate domain. 

\begin{lemma}[Curvature identities]
    Let $A = A_\bfa \, dx^\bfa$ be a connection $1$-form. We have the following identities, 
    \begin{enumerate}
        \item commutator identity,
            \begin{equation}\label{eq:commute}
                [\bfD_\bfa, \bfD_\bfb] 
                    = - i F_{\bfa \bfb},
            \end{equation}
        \item Bianchi identity,
            \begin{equation}\label{eq:bianchi}
                \partial_\bfa F_{\bfb \bfc} 
                    = \partial_\bfb F_{\bfa \bfc} - \partial_{\bfc} F_{\bfa \bfb}. 
        \end{equation}
    \end{enumerate}
    
\end{lemma}

\begin{proof}
    \leavevmode
    \begin{enumerate}
        \item Fix an arbitrary complex scalar field $\Phi$, then applying the product rule and commuting mixed partial derivatives, 
            \begin{align*}
                [\bfD_\bfa, \bfD_\bfb] \Phi
                    &= (\partial_\bfa - i A_\bfa) (\partial_\bfb - i A_\bfb)\Phi - (\partial_\bfb - i A_\bfb)(\partial_\bfa - i A_\bfa)\Phi \\
                    &= - i \left(\partial_\bfa (A_\bfb \Phi)  + A_\bfa \partial_\bfb \Phi - \partial_\bfb(A_\bfa \Phi) - A_\bfb \partial_\bfa \Phi \right) \\
                    &= - i (\partial_\bfa A_\bfb - \partial_\bfb A_\bfa) \Phi = - i F_{\bfa \bfb} \Phi.
            \end{align*}
        \item Commuting mixed partial derivatives, 
            \begin{align*}
            \partial_\bfa F_{\bfb \bfc} 
                &= \partial_\bfa (\partial_\bfb A_\bfc - \partial_\bfc A_\bfb) + (\partial_\bfc \partial_\bfb A_\bfa - \partial_\bfb \partial_\bfc A_\bfa)\\
                &= \partial_\bfb (\partial_\bfa A_\bfc - \partial_\bfc A_\bfa) - \partial_\bfc (\partial_\bfa A_\bfb - \partial_\bfb A_\bfa) = \partial_\bfb F_{\bfa \bfc} - \partial_{\bfc} F_{\bfa \bfb}.
        \end{align*}
    \end{enumerate}
\end{proof}

\begin{lemma}[Covariant Laplacian commutator identity]
    Let $A = A_\bfa \, dx^\bfa$ be a connection $1$-form, then 
        \begin{equation}\label{eq:laplace-commute}
            [\bfD_\bfa, \bfD^j \bfD_j]
                = -2i F\indices{_\bfa^j} \bfD_j - i \partial^j F_{\bfa j} 
        \end{equation}
\end{lemma}

\begin{proof}
   Fix an arbitrary complex scalar field $\Phi$, then commuting covariant derivatives a l\'a \eqref{commute} gives 
        \begin{align*}
            \bfD_\bfa \bfD^j \bfD_j \Phi
                &= \bfD^j \bfD_\bfa \bfD_j \Phi - i F\indices{_\bfa^j} \bfD_j \Phi \\
                &= \bfD^j \bfD_j \bfD_\bfa \Phi - i \bfD^j \big( F\indices{_\bfa_j} \Phi\big)- i F\indices{_\bfa^j} \bfD_j \Phi \\
                &= \bfD^j \bfD_j \bfD_\bfa \Phi - 2i F\indices{_\bfa^j} \bfD_j \Phi - i \partial^j F_{\bfa j} \Phi,
        \end{align*}
    which, upon rearranging terms, proves the identity. 
\end{proof}

\begin{lemma}[Covariant product rule]
    Let $A = A_\bfa \, dx^\bfa$ be a connection $1$-form, and let $\Phi, \Psi$ be complex scalar fields. Then we have the product rule,
        \begin{align}
            \partial_\bfa (\overline \Phi \Psi) 
                &= \overline{\bfD_\bfa \Phi} \Psi + \overline \Phi \bfD_\bfa \Psi. \label{eq:product}
        \end{align}
\end{lemma}

\begin{proof}
    By the usual product rule and remarking $\overline{(-i A_\bfa \Phi)} \Psi + \overline \Phi (-i A_\bfa \Psi) = 0$, we compute
        \begin{align*}
            \partial_\bfa (\overline \Phi \Psi) 
                &= \overline{\partial_\bfa \Phi} \Psi + \overline \Phi \partial_\bfa \Psi \\
                &=  \overline{\partial_\bfa \Phi} \Psi + \overline \Phi \partial_\bfa \Psi + \left( \overline{(-i A_\bfa \Phi)} \Psi + \overline \Phi (-i A_\bfa \Psi) \right) = \overline{\bfD_\bfa \Phi} \Psi + \overline \Phi \bfD_\bfa \Psi,
        \end{align*}
    which gives the result. 
\end{proof}

We state a standard pointwise inequality relating usual derivatives of the modulus $|\Phi|$ of a complex scalar field to covariant derivatives of the scalar field $\Phi$ itself. This will allow us to convert standard Sobolev inequalities into analogues for covariant derivatives.

\begin{lemma}[Diamagnetic inequality]
    Let $A = A_\bfa \, dx^\bfa$ be a connection $1$-form, and let $\Phi$ be a complex scalar field. Then the following inequality holds in distribution, 
        \begin{equation}
            |\partial_\bfa |\Phi|| \leq |\bfD_\bfa \Phi|.\label{eq:diamagnetic} 
        \end{equation}
\end{lemma}

\begin{proof}
    We give a formal calculation; writing $|\Phi|^2 = \overline \Phi \Phi$, we compute using the product rule \eqref{product} and Cauchy-Schwarz, 
        \begin{align*}
            |\partial_\bfa |\Phi||
                &= \frac{1}{\sqrt{\overline \Phi \Phi}} |\Re(\overline \Phi \bfD_\bfa \Phi)|\\
                &\leq |\bfD_\bfa \Phi|, 
        \end{align*}
    as desired. This argument may be made rigorous by regularising, replacing $|\Phi|$ with $\sqrt{|\Phi|^2 + \epsilon}$ and testing against non-negative test functions, upon which the result follows by taking $\epsilon \to 0$.
\end{proof}

\subsection{Magnetic Sobolev estimates}

To prove covariant energy estimates, we will need some analogues of Sobolev inequalities for magnetic derivatives $\bfD = \nabla - i A$,

\begin{lemma}[Magnetic Gagliardo-Nirenberg inequalities]
    Let $A = A_j \, dx^j$ be a connection $1$-form on $\R^2$. Then the following interpolation inequalities hold:
    \begin{enumerate}
        \item Ladyzhenskaya-type inequality; 
            \begin{equation}
                 \| \bfD \Phi \|_{L^4_x} 
                    \lesssim \|\Phi\|_{L^\infty_x}^{\frac12} \ \|\bfD^{(2)}\Phi\|_{L^2_x}^{\frac12},\label{eq:GN4} 
            \end{equation}

        \item Agmon-type inequality; 
            \begin{equation}
            \|\Phi\|_{L^\infty_x} 
                \lesssim \| \Phi \|_{L^2_x}^{\frac12} \ \bigl\| \bfD^{(2)} \Phi\bigr\|_{L^2_x}^{\frac12}. \label{eq:agmon}
            \end{equation}

        \item Gagliardo-Nirenberg-Sobolev inequality; for $2 \leq p < \infty$ and $\theta = \tfrac2p$, 
            \begin{equation}
                 \|\Phi \|_{L^p_x} 
                    \lesssim \|\Phi\|_{L^2_x}^{\theta} \ \|\bfD\Phi\|_{L^2_x}^{1 - \theta},\label{eq:GN} 
            \end{equation}

    \end{enumerate}
\end{lemma}

\begin{proof}
\leavevmode
    \begin{enumerate}
        \item Integrating-by-parts and applying Cauchy-Schwartz yields 
            \begin{align*}
                \|\bfD_j \Phi\|_{L^4_x}^4
                    &\approx \int_{\R^2} \bfD_j \Phi \cdot \bfD_j \Phi \cdot \bfD_j \Phi \cdot \bfD_j \Phi \, dx \\
                    &\approx \int_{\R^2} \Phi \cdot \bfD_j \Phi \cdot \bfD_j \Phi \cdot \bfD_j \bfD_j \Phi \, dx \lesssim \|\Phi\|_{L^\infty_x} \ \|\bfD_j \Phi\|_{L^4_x}^2 \ \|\bfD_j \bfD_j \Phi \|_{L^2_x}.
            \end{align*}
        Another application of Cauchy-Schwartz and rearranging gives
            \[
                \sum_{j = 1, 2}\| \bfD_j \Phi \|_{L^4_x}^4 
                    \lesssim \| \Phi\|_{L^\infty_x}^2 \, \sum_{j = 1, 2}\| \bfD_j \bfD_j \Phi \|_{L^2_x}^2. 
            \]

        \item Consider the following instance of the usual Gagliardo-Nirenberg inequality,
            \[
                \| w \|_{L^\infty_x} \lesssim \|w\|_{L^1_x}^{\frac13} \|\nabla^{(2)} w\|_{L^2_x}^{\frac23}.
            \]
        Applying this to the gauge-invariant quantity $|\Phi|^2$ yields
        \begin{align*}
            \|\Phi\|_{L^\infty_x}^2 
                &\lesssim \| |\Phi|^2 \|_{L^1_x}^{\frac13} \ \| \nabla^{(2)} |\Phi|^2 \|_{L^2_x}^{\frac23} \\
                &\lesssim \|\Phi\|_{L^2_x}^{\frac23} \left( \| \bfD \Phi\|_{L^4_x}^2 + \|\Phi\|_{L^\infty_x} \|\bfD^{(2)} \Phi\|_{L^2_x} \right)^{\frac23} \\
                &\lesssim \|\Phi\|_{L^\infty_x}^{\frac23} \ \|\Phi\|_{L^2_x}^{\frac23} \ \| \bfD^{(2)} \Phi\|_{L^2_x}^{\frac23},
        \end{align*}
    using the covariant product rule \eqref{product} in the second line, and the magnetic interpolation inequality \eqref{GN4} in the third line. Rearranging and taking appropriate roots gives the result. 
    
        \item This follows from the non-magnetic counterpart and the diamagnetic inequality \eqref{diamagnetic}. 
    \end{enumerate}
\end{proof}

We will also need an $L^\infty_x$-bound for scalar fields obeying the topological boundary condition $|\Phi| \to 1$ as $|x| \to \infty$. Given $\bfD \Phi \in L^2 (\R^2)$, it follows from the diamagnetic inequality \eqref{diamagnetic} that $|\Phi| \in \dot H^1 (\R^2)$; this is unfortunately not enough to deduce boundedness as the Sobolev inequality ${\dot H}^1 (\R^2)\not\hookrightarrow L^\infty (\R^2)$ fails by a logarithm. On the other hand, we only need mild additional assumptions at top-order and bottom-order, e.g.

\begin{lemma}[Sobolev inequality]
    Let $A \equiv A_j \, dx^j$ be a connection $1$-form on $\R^2$, then the following holds for any $\Phi \in L^\infty (\R^2)$ with $\bfD^{(2)} \Phi \in L^2 (\R^2)$ and charge density obeying $1 - |\Phi|^2 \in L^2 (\R^2)$,
        \begin{equation}\label{eq:inefficientLinfty}
            \|\Phi\|_{L^\infty_x} 
                \lesssim 1 + \big\|1 - |\Phi|^2\big\|_{L^2_x}^{\frac13} \ \big\|\bfD^{(2)} \Phi\big\|_{L^2_x}^{\frac13}.
        \end{equation}
\end{lemma}

\begin{proof}
    Consider the usual Agmon-type inequality,
        \[
            \| w\|_{L^\infty_x} 
                \lesssim \bigl\|w\bigr\|_{L^2_x}^{\frac12} \bigl\|\nabla^{(2)} w\bigr\|_{L^2_x}^{\frac12}
        \]
    Applying this to the gauge-invariant quantity $1 - |\Phi|^2$ yields
        \begin{align*}
            \big\| 1 - |\Phi|^2 \big\|_{L^\infty_x} 
                &\lesssim \big\|1 - |\Phi|^2 \big\|_{L^2_x}^{\frac12} \ \big\| \nabla^{(2)} |\Phi|^2 \big\|_{L^2_x}^{\frac12} \\
                &\lesssim \big\|1 - |\Phi|^2 \big\|_{L^2_x}^{\frac12} \left( \|\bfD \Phi\|_{L^4_x}^2 + \|\Phi\|_{L^\infty_x}  \ \big\|\bfD^{(2)} \Phi\big\|_{L^2_x} \right)^{\frac12}\\
                &\lesssim  \|\Phi\|_{L^\infty_x}^{\frac12} \ \big\|1 - |\Phi|^2 \big\|_{L^2_x}^{\frac12} \ \big\| \bfD^{(2)} \Phi\big\|_{L^2_x}^{\frac12},
        \end{align*}
   using the covariant product rule \eqref{product} in the second line, and the magnetic interpolation inequality \eqref{GN4} in the third line. Inserting this into a trivial triangle inequality bound gives
        \[
            \|\Phi\|_{L^\infty_x}^2 
                \lesssim 1 + \|\Phi\|_{L^\infty_x}^{\frac12} \ \big\|1 - |\Phi|^2 \big\|_{L^2_x}^{\frac12} \ \big\| \bfD^{(2)} \Phi\big\|_{L^2_x}^{\frac12}.
        \]
   Interpolating, we can absorb the $L^\infty_x$-term to the left-hand side, completing the proof. 
\end{proof}

\section{Topological degree}\label{sec:topology}

The most natural definition of degree arises from the topological boundary condition $|\phi| \to 1$ at spatial infinity $|x| \to \infty$. This condition is topological in the sense that, formally, we can define a \textit{boundary map at infinity} $\phi^\infty : \mathbb S^1 \to \mathbb S^1$ by identifying $\mathbb S^1$ with the unit circle in $\C$ and setting
    \[
        \phi^\infty (\theta) 
            := \lim_{r \to \infty} \phi(r e^{i \theta}).
    \]
When $\phi^\infty$ is well-defined and continuous, one can associate to it a topological degree $\deg[\phi^\infty]$ which is integer-valued. Jaffe-Taubes \cite[Chapter II, Conjecture I]{JaffeTaubes1980} conjectured that one could extend this notion to general finite-energy configurations -- this was answered in the affirmative, first by Rivi\'ere \cite[Section II]{Riviere2002}, and later approaches by Czubak-Jerrard \cite[Section 2]{CzubakJerrard2014} and Halavati \cite[Lemma 2.1]{Halavati2024}. 

We will adopt Czubak-Jerrard's approach, defining the \textit{vorticity},  
   \[
        \omega (A, \phi)
            := \epsilon_{jk} \partial^j \Bigl(A^k + \Im(\overline \phi \bfD^k \phi)\Bigr),
    \]
and the \textit{topological degree},
     \[
        \deg[A, \phi]
            := \tfrac{1}{2\pi}\int_{\R^2} \omega(A, \phi) \, dx.
    \]
These are gauge-invariant quantities which are well-defined for finite-energy configurations, 

\begin{lemma}[Properties of vorticity, I]\label{lem:vorticity-I}
    Let $(A, \phi)$ be a finite-energy configuration on $\R^2$. Then
    \begin{enumerate}
        \item the vorticity obeys the Bogomolnyi-type identity
            \begin{equation}\label{eq:vorticity-identity}
            \begin{split}
                \omega(A, \phi)
                    &=  (1 - |\phi|^2) F_{12} + 2 \Im(\overline{\bfD_1 \phi} \bfD_2 \phi),
            \end{split}
            \end{equation}

        \item the vorticity is $L^1_x$-integrable with
            \begin{equation}\label{eq:vorticity-energy}
                \| \omega(A, \phi) \|_{L^1_x} \lesssim \cE[A, \phi].
            \end{equation}
        Furthermore, the degree is well-defined and integer-valued $\deg(A, \phi) \in \Z$, and, when the boundary map at infinity $\phi^\infty : \mathbb S^1 \to \mathbb S^1$ is well-defined, agrees with the usual topological degree $\deg[A, \phi] = \deg[\phi^\infty]$. 

    \end{enumerate}
\end{lemma}

\begin{proof}
    \leavevmode
    \begin{enumerate}
        \item Differentiating and commuting appropriately yields
        \begin{align*}
            \omega(A, \phi)
                &= F_{12} + \epsilon_{jk} \partial^j \Im(\overline \phi \bfD^k \phi) \\
                &= F_{12} + \epsilon_{jk} \left( \Im(\overline{\bfD^j \phi} \bfD^k \phi) + \Im(\overline \phi \bfD^j \bfD^k  \phi) \right) \\
                &= (1 - |\phi|^2) F_{12} + 2 \Im(\overline{\bfD_1 \phi} \bfD_2 \phi),
        \end{align*}
        which yields \eqref{vorticity-identity}.

        \item The estimate \eqref{vorticity-energy} follows from the identity \eqref{vorticity-identity} and Cauchy-Schwartz. The integrality of the degree is due to Czubak-Jerrard \cite[Lemma 2.3]{CzubakJerrard2014}. 
        
    \end{enumerate}
\end{proof}

Our main observation here is the preservation of topological degree for $1$-parameter families of configurations $(A, \phi)$ which have uniformly bounded energy and are finite $L^2$-length,

\begin{proposition}[Properties of vorticity, II]\label{lem:vorticity-II}
    Let $(A, \phi)$ be a configuration on $I \times \R^2$ and set $A^{(h)} \equiv A_j (h, x) dx^j$ and $\phi^{(h)} \equiv \phi(h, x)$ such that $(A^{(h)}, \phi^{(h)})$ is a $1$-parameter family of configurations on $\R^2$ indexed by $h \in I$. Suppose that $(A, \phi)$ has finite $L^2$-length and the family $(A^{(h)}, \phi^{(h)})$ is uniformly bounded in energy, i.e.
    \[
        \sup_{h \in I} \cE[A^{(h)}, \phi^{(h)}] + \int_I \| \bfD_h \phi (h) \|_{L^2_x} + \sum_{j = 1, 2}\| F_{hj} \|_{L^2_x} \, dh < \infty.
    \]
Then 
    \begin{enumerate} 
        \item the vorticity obeys the density-flux identity
            \begin{equation}\label{eq:difference-flux}
                \partial_h \omega(A^{(h)}, \phi^{(h)})
                    = \epsilon_{jk} \partial^j \Bigl( (1 - |\phi|^2) F\indices{_h^k}  +  2 \Im(\overline{\bfD_h \phi} \bfD^k \phi)\Bigr),
            \end{equation}
        
        \item the degree of the family is preserved, i.e.
            \begin{equation}
                \deg[A^{(h_1)}, \phi^{(h_1)}] = \deg[A^{(h_2)}, \phi^{(h_2)}] \qquad \text{for all $h_1, h_2 \in I$}.
            \end{equation}
        
    \end{enumerate}
\end{proposition}

\begin{proof}
    \leavevmode
    \begin{enumerate}
        \item By the Bianchi identity \eqref{bianchi},
            \begin{align*}
                \partial_h \epsilon_{jk} \partial^j A 
                    = \partial_h F_{12} 
                    = \epsilon_{jk} \partial^j F\indices{_h^k} .
            \end{align*}
        Differentiating the charge current with respect to $h$ and commuting appropriately gives
            \begin{align*}
                \partial_h \epsilon_{jk} \partial^j \Im(\overline \phi \bfD^k \phi) 
                    &= \epsilon_{jk} \partial^j \Big( \Im(\overline{\bfD_h \phi} \bfD^k \phi) + \Im(\overline \phi \bfD_h \bfD^k \phi) \Big)\\
                    &= \epsilon_{jk} \partial^j \Big( \Im(\overline{\bfD_h \phi} \bfD^k \phi) + \Im(\overline \phi \bfD^k \bfD_h \phi) - |\phi|^2 F\indices{_h^k} \Big)\\
                    &= \epsilon_{jk} \partial^j \Big( 2\Im(\overline{\bfD_h \phi} \bfD^k \phi)  - |\phi|^2 F\indices{_h^k} \Big).
            \end{align*}
        Collecting the two calculations yields \eqref{difference-flux}.

        \item Fix a smooth cut-off $\chi \in C^\infty_c (\R^2)$ such that $\chi \equiv 1$ in a neighborhood of the origin, and test the identity \eqref{difference-flux} against its rescaling $\chi(\tfrac{x}{R})$ for $R > 0$. Integrating-by-parts on $[h_1, h_2] \times \R^2$ gives
            \begin{align*}
                \int_{\R^2} \omega(A^{(h)}, \phi^{(h)}) \, \chi\big(\tfrac{x}{R}\big) \, dx \Big|_{h = h_1}^{h = h_2}
                    &= -\frac1R\int_{h_1}^{h_2} \int_{\R^2}  \epsilon_{jk} (\partial^j \chi)\big(\tfrac{x}{R}\big) \Bigl( (1 - |\phi|^2) F\indices{_h^k}  +  2 \Im(\overline{\bfD_h \phi} \bfD^k \phi)\Bigr) \, dx dh.
            \end{align*}
        Taking $R \to \infty$, the expression on the left converges to the differences of the degrees by the dominated convergence theorem, 
            \[
                \lim_{R \to \infty} \frac{1}{2\pi}\text{L.H.S.}= \deg(A^{(h_2)}, \phi^{(h_2)}) - \deg(A^{(h_1)}, \phi^{(h_1)}).
            \]
        On the right-hand side, Holder's inequality yields 
            \begin{align*}
                \Big|\text{R.H.S.}\Big| 
                    &\lesssim \frac1R \Big(\sup_{h \in I} \cE[A^{(h)}, \phi^{(h)}]^{\frac12}\Big) \cdot  \Big(\int_I \| \bfD_h \phi (h) \|_{L^2_x} + \sum_{j = 1, 2}\| F_{hj} \|_{L^2_x} \, dh \Big) \longrightarrow 0.
            \end{align*}
        This proves the result. 
    \end{enumerate}
\end{proof}

\section{Self-dual structure}\label{sec:self-dual}

In this section, we reveal the self-dual structure of the problem by recasting the abelian Higgs energy and its gradient flow in terms of the \textit{covariant Cauchy-Riemann operators}
    \begin{align*}
        \bfD_z
            &= \tfrac12(\bfD_1 - i \bfD_2), \\
        \bfD_{\overline z}
            &= \tfrac12(\bfD_1 + i \bfD_2),
    \end{align*}
and the \textit{Gauss tension field},
    \begin{align*}
        \Gauss 
            := \tfrac12(1 - |\phi|^2) - F_{12}.
    \end{align*}

We begin by recalling Bogomoln'yi's \cite{Bogomolny1976} eponymous identity in our own notation -- given a finite-energy configuration $(A, \phi)$ on $\R^2$, we define its \textit{Bogomol'nyi energy} by
    \[
        \cE_{\, \overline\partial} [A, \phi]
            := \tfrac12 \int_{\R^2} \mathfrak |2 \bfD_{\overline z} \phi|^2 + |\Gauss|^2 \, dx.
    \]
\begin{proposition}[Bogomoln'yi identity]\label{prop:energyidentity}
	Let $(A, \phi)$ be a configuration on $\R^2$, then its self-dual abelian Higgs energy density may be recast in terms of its Bogomol'nyi energy density and vorticity by
		\begin{equation}\label{eq:Bogomolnyi-identity}
        \begin{split}
			 |\bfD_1 \phi|^2 + |\bfD_2 \phi|^2 + |F_{12}|^2 + \tfrac14 (1 - |\phi|^2)^2 
                &=  |2 \bfD_{\overline z} \phi|^2 + |\Gauss|^2 + \omega(A, \phi).
		\end{split}
        \end{equation}
    If $(A, \phi)$ is a finite-energy configuration, then 
        \begin{equation}\label{eq:Bogomolnyi-energy}
            \cE[A, \phi] 
                = \cE_{\, \overline\partial} [A, \phi] + \pi \deg [A, \phi].
        \end{equation}
\end{proposition}

\begin{proof}
	We expand the terms on the right-hand side,  
	    \[
            |2 \bfD_{\overline z} \phi|^2
                = \overline{\bfD^j \phi} \bfD_j \phi - 2 \Im(\overline{\bfD_1 \phi} \bfD_2 \phi), \label{eq:Dbar-square}
        \]
	and
		\[
			|\Gauss|^2 
                = |F_{12}|^2 + \tfrac14 (1 - |\phi|^2)^2 - (1 - |\phi|^2) F_{12} , \label{eq:Gauss-square} 
		\]
    Adding the expressions above with the vorticity identity \eqref{vorticity-identity} yields the identity \eqref{Bogomolnyi-identity}. Integrating and using quantisation of the vorticity yields the energy identity \eqref{Bogomolnyi-energy}. 
\end{proof}

This identity characterises the minimisers of the self-dual abelian Higgs energy within each (non-negative) topological class as solutions to the \textit{Bogomoln'yi equations} 
	\begin{equation}\label{eq:bogomolnyi}\tag{$\mathrm{B}$}
	\begin{split}
		(\bfD_1 + i \bfD_2) \phi 
			&= 0, \\
		F_{12} 
			&= \tfrac12(1 - |\phi|^2). 
	\end{split}
	\end{equation}
Remarkably, this first-order equation characterises all critical points of the self-dual abelian Higgs energy with non-negative topological degree. We will need this property, along with some qualitative properties of the charge, both of which are due to Jaffe-Taubes \cite[Chapter III, Theorem 1.1]{JaffeTaubes1980},

\begin{theorem}[Properties of vortices]\label{thm:JT}
    Let $(A, \phi) \in \cM^N$ be a finite-energy critical point of the self-dual abelian Higgs energy with non-negative topological degree $N$. Then $(A, \phi)$ is a solution to the Bogomoln'yi equations \eqref{bogomolnyi}. Furthermore, the magnetic field is non-negative and bounded, with 
        \begin{align}\label{eq:vortex-Linfty}
            0 \leq F_{12} \leq 1, 
        \end{align}
    and integrable $F_{12} \in L^1 (\R^2)$, with quantised integral,
        \begin{equation}\label{eq:quantised}
            \tfrac{1}{2\pi}\int_{\R^2} F_{12} \, dx 
                = N.
        \end{equation}
\end{theorem}

Another useful consequence of the Bogomoln'yi identity is that a configuration nearly minimises the self-dual abelian Higgs energy within its topological class if and only if the Bogomoln'yi energy is small. Thus, we can recast the assumptions of Theorem \ref{thm:main} as smallness of the \textit{Bogomoln'yi tension field} $(\Gauss, \bfD_{\overline z}\phi)$. With this in mind, we turn to revealing the self-dual structure of the gradient flow \eqref{AHG} by deriving the equations of motion for the Bogomoln'yi tension field. As a preliminary step, we derive the equations of motion for the magnetic field, charge density, and differentiated scalar field,

\begin{lemma}[Equations of motion]
    Let $(A, \phi)$ be a configuration on $[0, \infty) \times \R^2$ solving the self-dual abelian Higgs gradient flow \eqref{AHG}. Then the following covariant parabolic equations hold:
    \begin{itemize}
        \item the magnetic field, 
            \begin{equation}
                \bigl(\partial_t - \Delta + |\phi|^2 \bigr) F_{12} 
                    = 2 \Im(\overline{\bfD_1 \phi} \bfD_2 \phi), \label{eq:F12-parabolic}
            \end{equation}
        \item the charge density, 
            \begin{equation}
                \bigl(\partial_t - \Delta + |\phi|^2\bigr) \tfrac12 (1 - |\phi|^2) 
                    = \overline{\bfD^j \phi} \bfD_j \phi, \label{eq:charge-parabolic}
            \end{equation}

        \item the differentiated scalar field, 
            \begin{equation}
                \Bigl( \bfD_t - \bfD^j \bfD_j + |\phi|^2 \Bigr) \bfD_j \phi 
                    = \tfrac12(1 - |\phi|^2) \bfD_j \phi - 2i F_{12} \epsilon_{j\ell} \bfD^\ell \phi. \label{eq:Dphi-parabolic}
            \end{equation}
    \end{itemize}
\end{lemma}

\begin{proof}[Derivation of magnetic field equation \eqref{F12-parabolic}]
    Using the Bianchi identity \eqref{bianchi} and Ampere's law from \eqref{AHG}, the magnetic field obeys the equation
        \begin{align*}
            \partial_t F_{12} 
                &= \epsilon\indices{_j^k} \partial^j F_{tk} \\
                &= \epsilon\indices{_j^k} \bigl( \Im(\overline{\bfD^j \phi} \bfD_k \phi) + \Im(\overline \phi \bfD^j \bfD_k \phi) + \partial^j \partial^\ell F_{\ell k} \bigr) \\
                &= 2 \Im(\overline{\bfD_1 \phi} \bfD_2 \phi) + (\Delta - |\phi|^2) F_{12}, 
        \end{align*}
    as desired. 
\end{proof}

\begin{proof}[Derivation of charge density equation \eqref{charge-parabolic}]
Using the magnetic heat equation for the scalar field in \eqref{AHG}, the charge density obeys the equation
        \begin{align*}
            \partial_t \tfrac12(1 - |\phi|^2) 
                &= -\Re(\overline \phi \bfD_t \phi) \\
                &= -\Re\bigl(\overline \phi \bfD_j \bfD^j \phi\bigr) - |\phi|^2 \tfrac12(1 - |\phi|^2) \\
                &= \Re(\overline{\bfD^j \phi} \bfD_j \phi) + (\Delta - |\phi|^2) \tfrac12(1 - |\phi|^2) ,
        \end{align*}
    as desired. 
\end{proof}

\begin{proof}[Derivation of differentiated scalar field equation \eqref{Dphi-parabolic}]
    Using Ampere's law from \eqref{AHG} and the commutator \eqref{laplace-commute}, we have
        \begin{align*}
            \left[ \bfD_j, \bfD_t - \bfD^\ell \bfD_\ell - \tfrac12(1 - |\phi|^2) \right]
                &= i F_{tj} + 2 i F_{j\ell} \bfD^\ell + \Bigl( i \partial^\ell F_{j\ell} + \Re(\overline \phi \bfD_j \phi) \Bigr)\\
                &= i \Bigl( \Im(\overline \phi \bfD_j \phi) - \partial^\ell F_{j \ell} \Bigr) + 2i F_{12} \epsilon_{j \ell} \bfD^\ell + \Bigl( i \partial^\ell F_{j\ell} +  \Re(\overline \phi \bfD_j \phi)\Bigr) \\
                &= \overline \phi \bfD_j \phi +2 i F_{12} \epsilon_{j \ell} \bfD^\ell.
        \end{align*}
    Then the equation follows from differentiating the scalar field equation in \eqref{AHG} and applying the above commutator identity.
\end{proof}

We obtain the equations of motion for the Bogomoln'yi tension field by subtracting and adding the equations \eqref{F12-parabolic}-\eqref{Dphi-parabolic} appropriately. Miraculously, terms collect in such a way that the equations manifest as a closed system of massive parabolic equations with quadratic non-linearity,

\begin{proposition}[Equations for Bogomoln'yi tension fields]
    Let $(A, \phi)$ be a configuration on $[0, \infty) \times \R^2$ solving the abelian Higgs gradient flow \eqref{AHG}. Then the Bogomoln'yi tension field satisfies the system of massive parabolic equations
        \begin{align}
            \Bigl( \bfD_t - \bfD^j \bfD_j + \tfrac12(1 + |\phi|^2) \Bigr) \bfD_{\overline z} \phi 
                &= \Gauss \cdot 2 \bfD_{\overline z} \phi, \label{eq:Dbar-heat} \\
            \bigl(\partial_t - \Delta + |\phi|^2\bigr) \Gauss \label{eq:Gauss-heat}
                &= |2 \bfD_{\overline z} \phi|^2.
        \end{align}
\end{proposition}

\begin{proof}
    The Gauss tension field equation \eqref{Gauss-heat} is immediate from subtracting the magnetic field equation \eqref{F12-parabolic} from the charge density equation \eqref{charge-parabolic}, and observing the identity 
        \[
            |2 \bfD_{\overline z} \phi|^2
                = \overline{\bfD^j \phi} \bfD_j \phi - 2 \Im(\overline{\bfD_1 \phi} \bfD_2 \phi).
        \]
    For the differentiated Higgs field equation \eqref{Dbar-heat}, we add the equations \eqref{Dphi-parabolic} appropriately,  
        \begin{align*}
            \bigl( \bfD_t - \bfD^j \bfD_j + |\phi|^2 \bigr) \bfD_{\overline z} \phi 
                &= \tfrac12(1 - |\phi|^2) \bfD_{\overline z} \phi - 2 F_{12} \bfD_{\overline z} \phi \\
                &=  \tfrac12(|\phi|^2 - 1) \bfD_{\overline z} \phi + \Gauss \cdot 2 \bfD_{\overline z} \phi.
        \end{align*}
    Rearranging gives the result. 
\end{proof}

To transfer bounds for the Bogomoln'yi tension field to bounds for the configuration, we can recast the abelian Higgs gradient flow in self-dual form. That is, we rewrite the equation \eqref{AHG} in terms of the Bogomoln'yi tension field,

\begin{proposition}[Abelian Higgs gradient flow in self-dual form]
    Let $(A, \phi)$ be a configuration on $[0, \infty) \times \R^2$ solving the abelian Higgs gradient flow \eqref{AHG}. Then the equation may be expressed in self-dual form,
        \begin{align}
            \bfD_t \phi 
                &= 4 \bfD_z \bfD_{\overline z} \phi + \phi \Gauss, \label{eq:SD-AHG}\\
            F_{t \overline z}
                &= - i  \partial_{\overline z} \Gauss - i \overline \phi \bfD_{\overline z} \phi. \label{eq:SD-Amp}
        \end{align}
\end{proposition}

\begin{proof}[Derivation of Higgs field equation \eqref{SD-AHG}]
    We compute 
        \begin{align*}
        4 \bfD_z \bfD_{\overline z} 
            &= (\bfD_1 - i \bfD_2)(\bfD_1 + i \bfD_2) \\
            &= \bfD^j \bfD_j + i[\bfD_1, \bfD_2] \\
            &= \bfD^j \bfD_j + F_{12},
    \end{align*}
    Inserting this into the gradient flow equation \eqref{AHG} and collecting the charge density and magnetic field into the Gauss tension field gives
        \begin{align*}
            \bfD_t \phi 
                &= \Bigl( \bfD^j \bfD_j + \tfrac12(1 - |\phi|^2) \Bigr) \phi \\
                &= \Bigl( 4 \bfD_z \bfD_{\overline z} - F_{12} + \tfrac12(1 - |\phi|^2) \Bigr) \phi \\
                &= 4 \bfD_z \bfD_{\overline z} \phi + \phi \Gauss,
        \end{align*}
    as desired. 
\end{proof}

\begin{proof}[Derivation of Ampere's law \eqref{SD-Amp}]
    Using the Ampere's law \eqref{AHG} in $(x^1, x^2)$ coordinates, we compute 
        \begin{align*}
            2F_{t \overline z} 
                &= F_{t1} + i F_{t2} \\
                &= \Bigl(\Im(\overline \phi \bfD_1 \phi) + \Re(\overline \phi \bfD_2 \phi) + \partial_2 \Gauss \Bigr) + i \Bigl(\Im(\overline \phi \bfD_2 \phi) - \Re(\overline \phi \bfD_1 \phi) - \partial_1 \Gauss \Bigr) \\
                &= - i \Bigl(i \overline \phi \bfD_2 \phi + \overline \phi \bfD_1 \phi \Bigr) - i \Bigl( i \partial_2 \Gauss + \partial_1 \Gauss \Bigr)\\
                &= - 2 i \overline \phi \bfD_{\overline z} \phi - 2 i \partial_{\overline z} \Gauss.
        \end{align*}
    as desired. 
\end{proof}

\section{Covariant smoothing estimates}\label{sec:smooth}

We can collect the system of covariant parabolic equations \eqref{F12-parabolic}-\eqref{Dphi-parabolic} into the schematic form
    \begin{equation}\label{eq:schematic-parabolic-I}
        \bigl( \sfD_t - \sfD^j \sfD_j + 1 \bigr) \mathsf U 
            = \mathsf U \cdot \mathsf U,
    \end{equation}
where $\sfD = \nabla$ if applied to $\mathsf U \in \{ F_{12}, \tfrac12(1 - |\phi|^2) \}$, and $\sfD = \bfD$ if applied to $\mathsf U = \bfD \phi$. Similarly, the covariant parabolic equations for the Bogomoln'yi tension field \eqref{Dbar-heat}-\eqref{Gauss-heat} take the form
    \begin{equation}\label{eq:schematic-parabolic-II}
        \bigl( \sfD_t - \sfD^j \sfD_j + 1 \bigr) \mathsf B
            = \mathsf U \cdot \mathsf B 
    \end{equation}
where $\sfD = \nabla$ if applied to $\mathsf B = \Gauss$, and $\sfD = \bfD$ if applied to $\mathsf B = \bfD_{\overline z} \phi$. 

To facilitate some of the technical details in Section \ref{sec:converge}, we leverage the parabolic smoothing effects of the abelian Higgs gradient flow on unit-time scales. In summary, we regard the non-linear interactions as perturbative and show that the analogues of the parabolic smoothing estimates for the linear heat equation continue to hold for the parabolic equations \eqref{schematic-parabolic-I} and \eqref{schematic-parabolic-II} on the time-scale $t_* \equiv t_* (\cE)$, where we abbreviate the energy of the initial data by 
    \[
        \cE 
            := \cE[A^{\mathrm{in}}, \phi^{\mathrm{in}}],
    \]
and set 
    \[
        t_* 
            := c(1 + \cE)^{-2}, \qquad c \ll 1. 
    \]

Applying the energy method to \eqref{schematic-parabolic-I} gives

\begin{proposition}[Basic energy estimate]\label{prop:smooth-I}
    Let $(A, \phi)$ be a finite-energy configuration on $[0, \infty) \times \R^2$ solving the abelian Higgs gradient flow \eqref{AHG}. Then the following estimate holds for all $t_0 \geq 0$, 
        \begin{align}
            \bigl\| \mathsf U(t_0 + t) \bigr\|_{L^\infty_t L^2_x ([0, t_*])}^2 + \bigl\| \sfD \mathsf U (t_0 + t)\bigr\|_{L^2_{t, x} ([0, t_*])}^2
                &\lesssim \cE,\label{eq:smoothing} 
        \end{align}
    where $\sfD = \nabla$ if applied to $\mathsf U \in \{ F_{12}, \tfrac12(1 - |\phi|^2) \}$, and $\sfD = \bfD$ if applied to $\mathsf U = \bfD \phi$. 
\end{proposition}

\begin{proof}
    We multiply the equation \eqref{schematic-parabolic-I} by $\mathsf U$ and integrate-by-parts on $\R^2$. The mass term has favourable sign, so we can discard it, while on the right-hand side we use the Gagliardo-Nirenberg interpolation inequality \eqref{GN}, 
        \begin{align*}
            \frac{d}{dt} \tfrac12 \| \mathsf U \|_{L^2_x}^2 + \| \sfD \mathsf U \|_{L^2_x}^2 
                &\lesssim \Bigl|\bigl\langle \mathsf U, \mathsf U \cdot \mathsf U \bigr\rangle \Bigr|\\
                &\lesssim \| \mathsf U \|_{L^2_x}^2 \| \sfD \mathsf U \|_{L^2_x}.
        \end{align*}
    Integrating in time starting from $t = t_0$ up to $t = t_*$, then using Cauchy-Schwarz and energy monotonicity, 
        \begin{align*}
            \| \mathsf U \|_{L^\infty_t L^2_x}^2 + \| \sfD \mathsf U \|_{L^2_{t, x}}^2 
                &\lesssim \cE + |t_*|^{\frac12} \| \mathsf U\|_{L^\infty_t L^2_x}^2 \| \sfD \mathsf U \|_{L^2_{t, x}} .
        \end{align*}
    By a continuity argument, we can absorb the second term on the right-hand side into the left-hand side provided that $t_* \ll \cE^{-1}$, as desired. 
\end{proof}

As a useful corollary, we may deduce a unit-time scale $L^p_t L^\infty_x$-bound on the scalar field. The energy-class only provides $\dot H^1$-type control on $\phi$, so for fixed $t$ we cannot say that finite-energy configurations have bounded scalar field. Nonetheless, leveraging a little bit of averaging in $t$,

\begin{corollary}[$L^p_t L^\infty_x$-bound for $\phi$]\label{cor:Lp-Linfty}
    Let $(A, \phi)$ be a finite-energy configuration on $[0, \infty) \times \R^2$ solving the abelian Higgs gradient flow \eqref{AHG}. Then for $1 \leq p \leq 6$, the following bound holds for all $t_0 \geq 0$, 
        \begin{align} \label{eq:Lp-Linfty}
            \| \phi (t_0 + t) \|_{L^p_t L^\infty_x ([0, t_*])}
                \lesssim \bigl( 1 + \cE \bigr)^{-\frac2p}.
        \end{align}
\end{corollary}

\begin{proof}
    We have
        \begin{align*}
            \| \phi \|_{L^p_t L^\infty_x} 
                &\lesssim  \Bigl\| 1 + \bigl\| 1 - |\phi|^2 \bigr\|_{L^2_x}^{\frac13} \bigl\| \bfD^{(2)} \phi \bigr\|_{L^2_x}^{\frac13} \Bigr\|_{L^p_t} \\
                &\lesssim |t_*|^{\frac1p} + |t_*|^\frac1q \cE^{\frac16} \bigl\| \bfD^{(2)} \phi \bigr\|_{L^2_{t, x}}^\frac13 \\
                &\lesssim (1 + \cE)^{-\frac2p} ,
        \end{align*}
    where in the first line, we estimate the $L^\infty_x$-norm of the scalar field using the Sobolev-type embedding \eqref{inefficientLinfty}, in the second line we apply H\"older's inequality in time with $\tfrac1p = \tfrac16 + \tfrac1q$, and in the third line we use the parabolic smoothing estimate \eqref{smoothing}. 
\end{proof}

\begin{remark}
    The Sobolev embedding $\dot H^1 (\R^2) \not\hookrightarrow L^\infty (\R^2)$ only fails by a logarithm, so, working a little harder, one can prove a uniform $L^p_t L^\infty_x$-bound for the scalar field on unit time-scales for any finite $p$ by replacing the inefficient bound \eqref{inefficientLinfty} with a suitable logarithmic refinement of the Sobolev embedding. 
\end{remark}

With the basic energy estimate \eqref{smoothing} at hand, we may proceed inductively to prove higher-order smoothing estimates. The basis of our argument is the following weighted energy estimate,

\begin{lemma}[Abstract parabolic energy estimates]
    Let $\mathsf N : [t_1, t_2] \times \R^2 \to \C$ be a forcing term, and suppose that $\mathsf U : [t_1, t_2] \times \R^2 \to \C$ solves the parabolic equation
        \begin{equation}\label{eq:abstract-heat-eq}
            \big(\sfD_t - \sfD^j \sfD_j + 1 \big) \, \mathsf U
                = \mathsf N,
        \end{equation}
    where $\sfD = \nabla$ or $\bfD$. Then the following energy estimate holds for all exponents $\gamma \in \R$,
        \begin{equation}
             \big\| t^\gamma \mathsf U \big\|_{L^\infty_t L^2_x} + \big\| t^\gamma \sfD \mathsf U \big\|_{L^2_{t, x}}
                \lesssim \big\|t_1^\gamma \mathsf U (t_1)\big\|_{L^2_x} + \gamma \big\|t^{\gamma - \frac12} \mathsf U\big\|_{L^2_{t, x}} + \big\| t^{\gamma} \mathsf N \big\|_{L^1_t L^2_x}. \label{eq:abstract-heat-estimate}
        \end{equation}
\end{lemma}

\begin{proof}
    Commuting $t^\gamma$ into the abstract parabolic equation \eqref{abstract-heat-eq}, we arrive at 
        \[
            \big(\sfD_t - \sfD^j \sfD_j + 1\big) (t^\gamma \mathsf U) 
                = \gamma t^{\gamma - 1} \mathsf U + t^\gamma \mathsf N.
        \]
    Multiplying by $t^\gamma \mathsf U$, integrating on $\R^2$, and integrating-by-parts furnishes the energy identity
        \[
            \tfrac12 \partial_t \big\|t^\gamma \mathsf U\big\|_{L^2_x}^2 + \big\| t^\gamma \sfD \mathsf U \big\|_{L^2_x}^2 + \bigl\| t^\gamma \mathsf U\bigr\|_{L^2_x}^2  
                = \gamma \big\| t^{\gamma - \frac12} \mathsf U \big\|_{L^2_x}^2 + \int_{\R^2} t^\gamma \mathsf N \cdot t^\gamma \mathsf U \, dx. 
        \]
    Integrating forwards-in-time from $t = t_1$ and rearranging, we obtain
        \begin{align*}
            \tfrac12 \big\|t^\gamma \mathsf U\big\|_{L^\infty_t L^2_x}^2 + \big\| t^\gamma \sfD \mathsf U \big\|_{L^2_{t, x}}^2 
                &\leq \big\|t_1^\gamma \mathsf U(t_1)\big\|_{L^2_x}^2 + \gamma \big\|t^{\gamma - \frac12} \mathsf U\big\|_{L^2_{t, x}}^2 + \big\| t^{\gamma} \mathsf N \big\|_{L^1_t L^2_x}^2 + \tfrac14 \big\| t^\gamma \mathsf U \big\|_{L^\infty_t L^2_x}^2,
        \end{align*}
    removing the lower-order term on the left by virtue of the favourable sign, and bounding the right-hand side by the Cauchy-Schwarz inequality and Young's inequality. The last term on the right-hand side may be absorbed into the left-hand side, completing the proof. 
\end{proof}

We first apply this lemma to the higher-order analogues of \eqref{schematic-parabolic-I},

\begin{lemma}[Higher-order parabolic equations, I]\label{lem:higher-parabolic}
    Let $(A, \phi)$ be a finite-energy configuration on $[0, \infty) \times \R^2$ solving the self-dual abelian Higgs gradient flow \eqref{AHG}. Then for each non-negative integer $n \in \N_0$, we have the schematic equation 
        \begin{equation}\label{eq:schematic-higher-I}
        \begin{split}
            \bigl( \sfD_t - \sfD^j \sfD_j + 1 \bigr) \sfD^{(n)} \mathsf U  
                &= \sum_{a + b = n} \sfD^{(a)} \mathsf U \cdot \sfD^{(b)} \mathsf U \\
                    &\qquad +  \sum_{a + b + c = n - 2} \sfD^{(a)} \mathsf U \cdot \sfD^{(b)} \mathsf U \cdot \sfD^{(c)} \mathsf U  + \sum_{a + b = n - 1} \phi \cdot \sfD^{(a)} \mathsf U \cdot \sfD^{(b)} \mathsf U,
        \end{split}
        \end{equation}
    where $\sfD = \nabla$ if applied to $\mathsf U \in \{ F_{12}, \tfrac12(1 - |\phi|^2) \}$, and $\sfD = \bfD$ if applied to $\mathsf U = \bfD \phi$. 
\end{lemma}

\begin{proof}
    The case $n = 0$ is precisely \eqref{schematic-parabolic-I}; proceeding inductively, we assume the equation \eqref{schematic-higher-I} for $n$ and aim to show the $n + 1$ equation has appropriate form. Applying $\sfD$ to \eqref{schematic-higher-I}, the ensuing terms on the right-hand side arising from the product rule are acceptable, so it remains to verify that commuting $\sfD$ through the left-hand side is acceptable. The only non-trivial case is when $\mathsf U = \bfD \phi$ and $\sfD = \bfD$, for which we compute using Ampere's law \eqref{AHG}, 
        \begin{align*}
            \bigl[ \bfD_k, \bfD_t - \bfD^j \bfD_j \bigr] \bfD^{(n)} \bfD \phi
                &= \Bigl( i F_{tk}  + 2i F\indices{_k^j} \bfD_j + i \partial^j F_{kj} \Bigr) \bfD^{(n)} \bfD \phi \\
                &= i \Im(\overline \phi \bfD_k \phi) \bfD^{(n)} \bfD \phi + 2i F\indices{_k^j} \bfD_j \bfD^{(n)} \bfD \phi \\
                &= \phi \cdot \mathsf U \cdot \sfD^{(n)} \mathsf U + \mathsf U \cdot \sfD^{(n + 1)} \mathsf U . 
        \end{align*}
    This is acceptable in view of the right-hand side of \eqref{schematic-higher-I} for $n + 1$. 
\end{proof}

\begin{proposition}[Parabolic smoothing, I]\label{prop:smooth}
    Let $(A, \phi)$ be a finite-energy configuration on $[0, \infty) \times \R^2$ solving the abelian Higgs gradient flow \eqref{AHG}. For all $t_0 \geq 0$, non-negative integers $n \in \N_0$, and exponents $2 \leq p \leq \infty$, we have that
        \begin{equation}
            \bigl\| t^{\frac12 - \frac1p} (t^\frac12 \sfD)^{(n)} \mathsf U (t + t_0)  \bigr\|_{L^\infty_t L^p_x ([0, t_*])} + \bigl\| t^{\frac12 - \frac1p} (t^\frac12 \sfD)^{(n)} \sfD \mathsf U(t_0 + t) \bigr\|_{L^2_t L^p_x ([0, t_*])}
                \lesssim \cE^\frac12
        \end{equation}
    where $\sfD = \nabla$ if applied to $\mathsf U \in \{ F_{12}, \tfrac12(1 - |\phi|^2) \}$, and $\sfD = \bfD$ if applied to $\mathsf U = \bfD \phi$. 
\end{proposition}

\begin{proof}
    We begin with some reductions -- we can take, without loss of generality, $t_0 = 0$ and $p = 2$. The cases of general $t_0 \geq 0$ follows from time-translation symmetry, while the case $p = \infty$ would follow from the $p = 2$ case and the magnetic Gagliardo-Nirenberg interpolation inequality \eqref{agmon} and its non-magnetic counter-part. Interpolating between these two endpoints gives the intermediate exponents $2 < p < \infty$.

    Turning to the $L^2_x$-estimates, the case $n = 0$ is precisely the content of Proposition \ref{prop:smooth-I}. Proceeding from this base case, we argue by induction, assuming the result holds up to $n - 1$. Applying the abstract parabolic energy estimate \eqref{abstract-heat-estimate} with $\gamma = \tfrac{n}{2}$ to the schematic equation \eqref{schematic-higher-I} yields 
        \begin{align*}
            \bigl\| (t^\frac12 \sfD)^{(n)} \mathsf U \bigr\|_{L^\infty_t L^2_x} + \bigl\| (t^\frac12 \sfD)^{(n)} \sfD \mathsf U \bigr\|_{L^2_{t,x}} 
                &\lesssim \bigl\| (t^\frac12 \sfD)^{(n - 1)} \sfD \mathsf U \bigr\|_{L^2_{t, x}} + \bigl\| t^{\frac12} \text{R.H.S. \text{\eqref{schematic-higher-I}}} \bigr\|_{L^1_t L^2_x} \\
                &\lesssim \cE^\frac12 + \bigl\| t^{\frac12} \text{R.H.S. \text{\eqref{schematic-higher-I}}} \bigr\|_{L^1_t L^2_x},
        \end{align*}
    applying the inductive hypothesis to handle the $L^2_{t, x}$-term on the right-hand side. It remains to estimate the right-hand side of \eqref{schematic-higher-I}. We make the bootstrap assumption that the left-hand side is $O_n (\cE^{1/2})$ for sufficiently large constant. 

    The quadratic terms on the right-hand side of \eqref{schematic-higher-I} may be estimated by 
        \begin{align*}
            \int_0^{t_*} t^{\frac{n}{2}} \bigl\| \sfD^{(a)} \mathsf U \cdot \sfD^{(b)} \mathsf U \bigr\|_{L^2_x} \, dt 
                &\lesssim \int_0^{t_*} \bigl\|(t^\frac12 \sfD)^{(a)} \mathsf U \bigr\|_{L^4_x} \, \bigl\|(t^\frac12 \sfD)^{(b)} \mathsf U \bigr\|_{L^4_x} \, dt \\
                &\lesssim \int_0^{t_*} \prod_{m = a, b} \bigl\| (t^\frac12 \sfD)^{(m)} \mathsf U \bigr\|_{L^2_x}^\frac12 \bigl\| (t^\frac12 \sfD)^{(m)} \sfD \mathsf U  \bigr\|_{L^2_x}^{\frac12} \, dt\\
                &\lesssim |t_*|^{\frac12} \prod_{m = a, b} \bigl\| (t^\frac12 \sfD)^{(m)} \mathsf U \bigr\|_{L^\infty_t L^2_x}^\frac12 \bigl\| (t^\frac12 \sfD)^{(m)} \sfD \mathsf U  \bigr\|_{L^2_{t, x}}^{\frac12} \\
                &\lesssim |t_*|^{\frac12} \cE, 
        \end{align*}
    for $a + b = n$; we have applied Cauchy-Schwarz in $x$ in the first line, Gagliardo-Nirenberg in the second line, H\"older in $t$ in the third line, and either the bootstrap assumption or the induction hypothesis in the last line. Taking $|t_*| \ll \cE^{-1}$ as assumed allows us to regard this as perturbative. 

    Arguing similarly for the first set of cubic terms in \eqref{schematic-higher-I} gives 
        \begin{align*}
            \int_0^{t_*} t^{\frac{n}{2}} \bigl\| \sfD^{(a)} \mathsf U \cdot \sfD^{(b)} \mathsf U \cdot \sfD^{(c)} \mathsf U \bigr\|_{L^2_x} \, dt 
                &\lesssim  \int_0^{t_*}  \bigl\| t^\frac12 (t^\frac12 \sfD)^{(a)} \mathsf U \bigr\|_{L^2_x} \bigl\| t^\frac12 (t^\frac12 \sfD)^{(b)} \mathsf U \bigr\|_{L^\infty_x} \bigl\| (t^\frac12 \sfD)^{(c)} \mathsf U \bigr\|_{L^\infty_x} \, dt \\
                &\lesssim \int_0^{t_*} \bigl\| (t^\frac12 \sfD)^{(a)} \mathsf U \bigr\|_{L^2_x}  \prod_{m = b, c} \bigl\| (t^\frac12 \sfD)^{(m)} \mathsf U \bigr\|_{L^2_x}^{\frac12} \bigl\| (t^\frac12 \sfD)^{(m + 1)} \sfD \mathsf U \bigr\|_{L^2_x}^{\frac12} \, dt \\
                &\lesssim |t_*|^{\frac12} \bigl\| (t^\frac12 \sfD)^{(a)} \mathsf U \bigr\|_{L^\infty_t L^2_x} \prod_{m = b, c} \bigl\| (t^\frac12 \sfD)^{(m)} \mathsf U \bigr\|_{L^\infty_t L^2_x}^{\frac12} \bigl\| (t^\frac12 \sfD)^{(m + 1)} \sfD \mathsf U \bigr\|_{L^2_{t, x}}^{\frac12} \\
                &\lesssim |t_*|^{\frac12} \cE^{\frac32}
        \end{align*}
    for $a + b + c = n - 2$. This is the most dangerous term, requiring us to take $|t_*| \ll \cE^{-2}$.
        
    For the second set of cubic terms in \eqref{schematic-higher-I}, we proceed similarly, with the addition of estimating the scalar field using the $L^2_t L^\infty_x$-bound \eqref{Lp-Linfty},
        \begin{align*}
            \int_0^{t_*} t^{\frac{n}{2}} \bigl\| \phi \cdot \sfD^{(a)} \mathsf U \cdot \sfD^{(b)} \mathsf U \bigr\|_{L^2_x} \, dt
                &\lesssim \int_0^{t_*} t^\frac12 \|  \phi \|_{L^\infty_x} \bigl\| (t^\frac12 \sfD)^{(a)} \mathsf U \bigr\|_{L^4_x} \bigl\| (t^\frac12 \sfD)^{(b)} \mathsf U \bigr\|_{L^4_x} \, dt \\
                &\lesssim \int_0^{t_*} t^{\frac12} \|\phi \|_{L^\infty_x} \prod_{m = a, b} \bigl\| (t^\frac12 \sfD)^{(m)} \mathsf U \bigr\|_{L^2_x}^\frac12 \bigl\| (t^\frac12 \sfD)^{(m)} \sfD \mathsf U  \bigr\|_{L^2_x}^{\frac12} \, dt\\
                &\lesssim |t_*|^{\frac12} \| \phi \|_{L^2_t L^\infty_x} \prod_{m = a, b} \bigl\| (t^\frac12 \sfD)^{(m)} \mathsf U \bigr\|_{L^\infty_t L^2_x}^\frac12 \bigl\| (t^\frac12 \sfD)^{(m)} \sfD \mathsf U  \bigr\|_{L^2_{t, x}}^{\frac12} \\
                &\lesssim |t_*|^{\frac12}, 
        \end{align*}
    for $a + b = n - 1$. There is considerable room here, so this contribution is easily seen to be acceptable. We conclude the induction and thereby the proof. 
\end{proof}

Using the smoothing estimates for the $\mathsf U$ variables, we can prove smoothing estimates for the variables $\mathsf B$. Indeed, the equation \eqref{schematic-parabolic-II} is roughly-speaking a linearisation of \eqref{schematic-parabolic-I}. We will follow a similar argument as in Proposition \ref{prop:smooth}; for this reason, we paint the proof here in broad strokes, and leave some of the intermediary justifications to the reader. 

\begin{lemma}[Higher-order parabolic equations, II]
    Let $(A, \phi)$ be a configuration on $[0 ,\infty) \times \R^2$ solving the self-dual abelian Higgs gradient flow \eqref{AHG}. Then, for each non-negative integer $n \in \N_0$, the following covariant parabolic equation holds, 
        \begin{equation}
        \begin{split}
            \bigl(\sfD_t - \sfD^j \sfD_j + 1\bigr) \sfD^{(n)} \mathsf B
                &= \sum_{a + b = n} \sfD^{(a)} \mathsf U \cdot \sfD^{(b)} \mathsf B \\
                    &\qquad +  \sum_{a + b + c = n - 2} \sfD^{(a)} \mathsf U \cdot \sfD^{(b)} \mathsf U \cdot \sfD^{(c)} \mathsf B  + \sum_{a + b = n - 1} \phi \cdot \sfD^{(a)} \mathsf U \cdot \sfD^{(b)} \mathsf B,
        \end{split}\label{eq:schematic-higher-II}
        \end{equation}
    where $\sfD = \nabla$ if applied to $\mathsf B = \Gauss$ and $\mathsf U \in \{ F_{12}, \tfrac12(1 - |\phi|^2) \}$, and $\sfD = \bfD$ if applied to $\mathsf B = \bfD_{\overline z} \phi$ and $\mathsf U = \bfD \phi$. 
\end{lemma}

\begin{proof}
    The proof is almost verbatim that of Lemma \ref{lem:higher-parabolic}, using \eqref{schematic-parabolic-II} as the base case $n = 0$. 
\end{proof}

\begin{proposition}[Parabolic smoothing, II]\label{prop:smooth-II}
    Let $(A, \phi)$ be a finite-energy configuration on $[0, \infty) \times \R^2$ solving the self-dual abelian Higgs gradient flow \eqref{AHG}. For all $t_0 \geq 0$, non-negative integers $n \in \N_0$, and exponents $2 \leq p \leq \infty$, we have that 
        \begin{equation}\label{eq:B-smoothing}
            \bigl\| t^{\frac12 - \frac1p} (t^\frac12 \sfD)^{(n)} \mathsf B (t + t_0)  \bigr\|_{L^\infty_t L^p_x ([0, t_*])} + \bigl\| t^{\frac12 - \frac1p} (t^\frac12 \sfD)^{(n)} \sfD \mathsf B(t_0 + t) \bigr\|_{L^2_t L^p_x ([0, t_*])}
                \lesssim \| \mathsf B (t_0) \|_{L^2_x},
        \end{equation}
    where $\sfD = \nabla$ if applied to $\mathsf B = \Gauss$, and $\sfD = \bfD$ if applied to $\mathsf B = \bfD_{\overline z} \phi$.
\end{proposition}

\begin{proof}
    It will suffice to consider $t = 0$ and $p = 2$. We assume for induction the result up to $n - 1$. Applying the abstract parabolic energy estimate \eqref{abstract-heat-estimate} with $\gamma = \tfrac{n}{2}$ to the schematic equation \eqref{schematic-higher-II} yields
        \begin{align*}
            \bigl\| (t^\frac12 \sfD)^{(n)} \mathsf B \bigr\|_{L^\infty_t L^2_x} + \bigl\| (t^\frac12 \sfD)^{(n)} \sfD \mathsf B \bigr\|_{L^2_{t,x}} 
                &\lesssim \bigl\| (t^\frac12 \sfD)^{(n - 1)} \sfD \mathsf B \bigr\|_{L^2_{t, x}} + \bigl\| t^{\frac12} \text{R.H.S. \text{\eqref{schematic-higher-II}}} \bigr\|_{L^1_t L^2_x} \\
                &\lesssim \| \mathsf B^{\mathrm{in}} \|_{L^2_x} + \bigl\| t^{\frac12} \text{R.H.S. \text{\eqref{schematic-higher-II}}} \bigr\|_{L^1_t L^2_x},
        \end{align*}
    applying the inductive hypothesis to handle the $L^2_{t, x}$-term on the right-hand side. It remains to estimate the right-hand side of \eqref{schematic-higher-I}. We make the bootstrap assumption that the left-hand side is $O_n (\| \mathsf B^{\mathrm{in}} \|_{L^2_x})$ for sufficiently large constant. 
        
    The quadratic terms on the right-hand side of \eqref{schematic-higher-I} may be estimated by 
        \begin{align*}
            \int_0^{t_*} t^{\frac{n}{2}} \bigl\| \sfD^{(a)} \mathsf U \cdot \sfD^{(b)} \mathsf B \bigr\|_{L^2_x} \, dt 
                &\lesssim |t_*|^{\frac12} \cE^\frac12 \| \mathsf B^{\mathrm{in}} \|_{L^2_x}, 
        \end{align*}
    for $a + b = n$, the first set of cubic terms in \eqref{schematic-higher-II} by 
        \begin{align*}
            \int_0^{t_*} t^{\frac{n}{2}} \bigl\| \sfD^{(a)} \mathsf U \cdot \sfD^{(b)} \mathsf U \cdot \sfD^{(c)} \mathsf B \bigr\|_{L^2_x} \, dt 
                &\lesssim |t_*|^{\frac12} \cE \| \mathsf B^{\mathrm{in}} \|_{L^2_x},
        \end{align*}
    for $a + b + c = n - 2$, and the second set of cubic terms in \eqref{schematic-higher-II} by 
        \begin{align*}
            \int_0^{t_*} t^{\frac{n}{2}} \bigl\| \phi \cdot \sfD^{(a)} \mathsf U \cdot \sfD^{(b)} \mathsf B \bigr\|_{L^2_x} \, dt
                &\lesssim |t_*|^{\frac12} (1 + \cE)^{-1} \cE^{\frac12} \| \mathsf B^{\mathrm{in}} \|_{L^2_x}, 
        \end{align*}
    for $a + b = n - 1$. Taking $|t_*| = c (1 + \cE)^{-2}$ for some $c \ll 1$ and using the smoothing bounds \eqref{smoothing} from Proposition \ref{prop:smooth} gives the result. 
\end{proof}

\section{Convergence of the flow}\label{sec:converge}

Define the \textit{Bogomoln'yi tension field} by 
    \[
        \Bogomolnyi
            := (\Gauss, 2\bfD_{\overline z} \phi). 
    \]
One should view the Bogomoln'yi tension field as measuring the extent to which a configuration fails to satisfy the Bogomoln'yi equations \eqref{bogomolnyi}. Our strategy for proving Theorem \ref{thm:main} consists of two steps, (\texttt{i}) prove covariant decay estimates for the Bogomoln'yi tension field, (\texttt{ii}) convert the decay estimates into difference estimates for the configuration by recasting them in terms of the Bogomoln'yi tension field. The Bogomoln'yi identity \eqref{Bogomolnyi-energy} implies that
    \[
         \| \Bogomolnyi \|_{L^2_x}^2
            \sim \cE[A, \phi] - \pi \deg [A, \phi],
    \]
so hereon, we restate the smallness assumption in Theorem \ref{thm:main} in terms of the Bogomoln'yi tension field.

\subsection{Decay of Bogomoln'yi tension field}\label{sec:decay}

Packaging together the equations \eqref{Dbar-heat}-\eqref{Gauss-heat}, the Bogomoln'yi tension field satisfies a massive parabolic equation with quadratic non-linearity, namely
    \begin{equation}\label{eq:schematic-B}
        \bigl(\sfD_t - \sfD^j \sfD_j + \mathsf m(\phi) \bigr) \Bogomolnyi 
            = \Bogomolnyi \cdot \Bogomolnyi,
    \end{equation}
where the mass is
    \[
        \mathsf m(\phi) 
            := \bigl(|\phi|^2, \tfrac12(1 + |\phi|^2)\bigr),
    \]  
and 
    \[ 
        \mathsf D 
            := (\nabla, \bfD),
    \] 
acts on the Bogomoln'yi tension field $\Bogomolnyi$ in a component-wise fashion. One should view the presence of mass in \eqref{schematic-B} as a manifestation of the eponymous \textit{Higgs mechanism} in this particular problem. The model equation one should compare with is the massive semi-linear heat equation
    \[
        (\partial_t - \Delta + 1) u = u^2.
    \]
It is not difficult to show using an energy argument that solutions to the model equation with small initial data in $L^2 (\R^2)$ decay exponentially in time. With some minor modifications, one can run the same argument on the system of equations \eqref{schematic-B} to show 

\begin{proposition}[Exponential decay of tension field]\label{prop:small-decay}
    Let $N \in \N_0$ be a non-negative integer, and suppose that $\epsilon_* \ll_N 1$ and $\gamma \ll_N 1$ are sufficiently small. Consider a finite-energy configuration $(A^{\mathrm{in}}, \phi^{\mathrm{in}})$ on $\R^2$ with topological degree $N$ and self-dual energy satisfying
        \begin{equation}\label{eq:small-tension}
            \| \Bogomolnyi^{\mathrm{in}} \|_{L^2_x}^2 
                \leq \epsilon_*.
        \end{equation}
    Then the solution $(A, \phi)$ to the self-dual abelian Higgs gradient flow \eqref{AHG} on $[0, \infty) \times \R^2$ with initial data $(A, \phi)_{|t = 0} = (A^{\mathrm{in}}, \phi^{\mathrm{in}})$ satisfies the decay estimate
        \begin{equation}\label{eq:bogomolnyi-decay}
            \sup_{t \in [0, \infty)} e^{\gamma t} \| \Bogomolnyi \|_{L^2_x}^2 + \frac{1}{1 + N} \int_0^\infty e^{\gamma t} \Bigl(\| \sfD \Bogomolnyi \|_{L^2_x}^2 + \| \Bogomolnyi \|_{L^2_x}^2 \Bigr) \, dt
            \lesssim \| \Bogomolnyi^{\mathrm{in}} \|_{L^2_x}^2 . 
    \end{equation}
\end{proposition}

The proof follows a standard energy and continuity argument, modulo a technical coercivity lemma. Testing \eqref{schematic-B} against $\Bogomolnyi$ and integrating-by-parts, we obtain the differential identity 
    \begin{align}\label{eq:diff-identity1}
        \frac{d}{dt} \bigl\| \Bogomolnyi \bigr\|_{L^2_x}^2 + \Big\langle  \Bogomolnyi, \bigl( - \sfD^j \sfD_j + \mathsf m (\phi)\bigr) \Bogomolnyi \Big\rangle_{L^2}
            = \int_{\R^2} \Bogomolnyi \cdot \Bogomolnyi \cdot \Bogomolnyi \, dx. 
    \end{align}
Ignoring the non-linear terms on the right-hand side, to gain exponential decay in time, it would follow from Gronwall's inequality provided that the bilinear form
    \[
        (u, v) \mapsto \Big\langle  u, \bigl( - \sfD^j \sfD_j + \mathsf m (\phi)\bigr) v \Big\rangle_{L^2}
    \]
is coercive on $L^2 (\R^2 \to \R \times \C)$. The mass $\tfrac12(1 + |\phi|^2)$ within the scalar field equation \eqref{Dbar-heat} is uniformly bounded away from zero, so coercivity for this component is immediate. However, the mass $|\phi|^2$ within the Gauss tension field equation \eqref{Gauss-heat} may vanish or become arbitrarily small. Nonetheless, the finite-energy condition implies that this mass is non-trivial in an averaged sense, allowing us to prove

\begin{lemma}[Coercivity lemma]\label{lemma:coercive}
    Let $(A, \phi) \in \frE$ be a finite-energy configuration. Then, for $u \in H^1 (\R^2 \to \C)$, 
        \begin{equation}\label{eq:coercive-I}
            \|\nabla u\|_{L^2_x} + \| u \|_{L^2_x}^2
                \lesssim (1 + \cE[\phi]) \, \Bigl\langle u, \bigl(- \Delta + |\phi|^2\bigr) u \Bigr\rangle,
        \end{equation}
    and 
        \begin{equation}\label{eq:coercive-II}
            \| \bfD u \|_{L^2_x}^2 + \| u \|_{L^2_x}^2
                \lesssim \Bigl\langle u, \bigl(-\bfD^j \bfD_j + \tfrac12(1 + |\phi|^2)\bigr) u\Bigr\rangle.
        \end{equation}
\end{lemma}

\begin{proof}
    The second inequality is immediate from integration-by-parts. For the first inequality, integrating-by-parts and rewriting $|\phi|^2 = 1 - (1 - |\phi|^2)$, we obtain
        \begin{align*}
            \| \nabla u \|_{L^2_x}^2 + \| |\phi| u \|_{L^2_x}^2 
                = \bigl\langle u, (|\phi|^2 - \Delta) u \bigr\rangle
                = \| u \|_{H^1_x}^2 - \int_{\R^2} (1 - |\phi|^2) |u|^2 \, dx.
        \end{align*}
    For the error term on the right-hand side, it follows from Cauchy-Schwartz and Gagliardo-Nirenberg that
        \begin{align*}
            \Bigl|\int_{\R^2} (1 - |\phi|^2) |u|^2 \, dx  \Bigr|
                &\leq \bigl\| 1 - |\phi|^2 \bigr\|_{L^2_x} \, \| u \|_{L^4_x}^2  \\
                &\leq C \, \cE[\phi]^{\frac12} \, \|u \|_{L^2_x} \, \| \nabla u \|_{L^2_x} \\
                &\leq \frac{1}{4} \| u \|_{L^2_x}^2 + C^2 \, \cE[\phi] \, \| \nabla u\|_{L^2_x}^2 ,
        \end{align*}
    for some constant $C > 0$. Collecting the two calculations, we have 
        \begin{align*}
            \| u \|_{H^1_x}^2 
                &\leq \| \nabla u \|_{L^2_x}^2 + \| |\phi| u \|_{L^2_x}^2 + \Bigl|\int_{\R^2} (1 - |\phi|^2) |u|^2 \, dx  \Bigr| \\
                &\leq \Bigl( 1 + C^2 \cdot \cE[\phi] \Bigr) \Bigl(\| \nabla u \|_{L^2_x}^2 + \| |\phi| u \|_{L^2_x}^2 \Bigr) + \frac14 \| u \|_{L^2_x}^2 \\
                &\leq \Bigl( 1 + C^2 \cdot \cE[\phi] \Bigr) \bigl\langle u, (|\phi|^2 - \Delta) u \bigr\rangle + \frac14 \| u \|_{H^1_x}^2.
        \end{align*}
    Moving the second term on the right-hand side to the left, we conclude the result. 
\end{proof}

\begin{remark}
    It is clear from the proof that we can relax the finite-energy assumption and replace the abelian Higgs energy $\cE[A, \phi]$ with the potential energy of the scalar field, 
        \[
            \cV[\phi]
                := \tfrac12 \int_{\R^2} \tfrac14 (1 - |\phi|^2)^2 \, dx. 
        \]
\end{remark}

Returning to the differential identity \eqref{diff-identity1}, applying the coercivity estimate \eqref{coercive-I}-\eqref{coercive-II} to the left-hand side and Gagliardo-Nirenberg interpolation \eqref{GN} to the right-hand side yields the differential inequality
    \[
        \frac{d}{dt} \bigl\| \Bogomolnyi \bigr\|_{L^2_x}^2 + \frac1C \frac{1}{1 +  N} \Bigl( \| \sfD \Bogomolnyi \|_{L^2_x}^2 + \| \Bogomolnyi \|_{L^2_x}^2 \Bigr) 
            \leq C \| \Bogomolnyi \|_{L^2_x}^2 \| \sfD \Bogomolnyi \|_{L^2_x},
    \]
for some uniform constant $C > 1$. Commuting the exponential weight $e^{\gamma t}$ for, say, $\gamma := \tfrac{1}{2C} \tfrac{1}{2 + \pi N}$, 
    \begin{align*}
        \frac{d}{dt} \Bigl( e^{\gamma t} \| \Bogomolnyi \|_{L^2_x}^2 \Bigr) + \frac1{2C} \frac{1}{2 + \pi N} e^{\gamma t}\Bigl( \| \sfD \Bogomolnyi \|_{L^2_x}^2 + \| \Bogomolnyi \|_{L^2_x}^2 \Bigr) \leq C  e^{\gamma t}  \| \Bogomolnyi \|_{L^2_x}^2 \| \sfD \Bogomolnyi \|_{L^2_x},
    \end{align*}
absorbing the commutator with the bad sign into the massive bulk. Integrating this differential inequality on $[0, T]$ gives the energy estimate
    \begin{equation}\label{eq:non-linear-decay}
        \begin{split}
        \sup_{t \in [0, T]} e^{\gamma t} \| \Bogomolnyi \|_{L^2_x}^2 &+ \frac1C \frac{1}{1 + N}\int_0^T e^{\gamma t} \Bigl(\| \sfD \Bogomolnyi \|_{L^2_x}^2 + \| \Bogomolnyi \|_{L^2_x}^2 \Bigr) \, dt \\
            &\qquad \leq \| \Bogomolnyi^{\mathrm{in}} \|_{L^2_x}^2 + C \int_0^T e^{\gamma t} \| \Bogomolnyi \|_{L^2_x}^2 \| \sfD \Bogomolnyi \|_{L^2_x} \, dt,
        \end{split}
    \end{equation}
adjusting the uniform constant $C > 1$ appropriately. 

To conclude the proof, we argue by continuous induction on $T$, setting 
    \[
        \epsilon := \| \Bogomolnyi^{\mathrm{in}} \|_{L^2_x}^2,
    \]
and making the bootstrap assumption 
    \begin{equation}\label{eq:bootstrap-decay}
        \sup_{t \in [0, T]} e^{\gamma t} \| \Bogomolnyi \|_{L^2_x}^2 + \frac1C \frac1{1 + N} \int_0^T e^{\gamma t} \Bigl(\| \sfD \Bogomolnyi \|_{L^2_x}^2 + \| \Bogomolnyi \|_{L^2_x}^2 \Bigr) \, dt
            \leq 2 \epsilon  . 
    \end{equation}
Evidently this holds for $T \ll 1$ by the assumption on the initial data, so it remains to improve the bootstrap assumption to close the continuity argument and conclude \eqref{bootstrap-decay} for all $T > 0$ and thereby the proposition. We proceed by estimating the right-hand side of the energy estimate \eqref{non-linear-decay} using the bootstrap assumption \eqref{bootstrap-decay}, 
    \begin{align*}
        C \int_0^T e^{\gamma t} \| \Bogomolnyi \|_{L^2_x}^2 \| \sfD \Bogomolnyi \|_{L^2_x} \, dt 
            &\leq \frac{C}{2}  \sup_{t \in [0, T]} \| \Bogomolnyi\|_{L^2_x} \cdot \int_0^T e^{\gamma t}  \Bigl( \| \sfD \Bogomolnyi \|_{L^2_x}^2 + \| \Bogomolnyi \|_{L^2_x}^2 \Bigr) \, dt \\
            &\leq C^2 (1 + N) \epsilon^{\frac32} \\
            &\leq \frac12 \epsilon, 
    \end{align*}
choosing $\epsilon_* := \tfrac1{4C^4} \tfrac1{(1 + N)^2}$. Inserting this estimate and the smallness assumption on the initial data \eqref{small-tension} into the energy estimate \eqref{non-linear-decay}, we improve the bootstrap assumption \eqref{bootstrap-decay} to
    \begin{align*}
        \sup_{t \in [0, T]}\| e^{\gamma t} \Bogomolnyi \|_{L^2_x}^2 + \frac1C \frac1{1 + N} \int_0^T e^{\gamma t} \Bigl(\| \sfD \Bogomolnyi \|_{L^2_x}^2 + \| \Bogomolnyi \|_{L^2_x}^2 \Bigr) \, dt 
            \leq \frac32 \epsilon,
    \end{align*}
closing the continuity argument and thereby the proof.

\subsection{$(H^1\times L^2)$-convergence of $(A, \phi)$}\label{sec:H1-L2}

Imposing the temporal gauge $A_t = 0$, the abelian Higgs gradient flow in self-dual form \eqref{SD-AHG}-\eqref{SD-Amp} can be put in the schematic form 
    \begin{equation}\label{eq:schematic-SD-AHG}
        \partial_t (A, \phi) 
            = \phi \cdot \Bogomolnyi + \sfD \Bogomolnyi. 
    \end{equation}
Using this identity, we can easily convert decay of the Bogomoln'yi tension field into difference bounds for the configuration. To handle various coefficients, such as the scalar field $\phi$, which appear in the calculations, it will be convenient to record the following interpolation lemma,

\begin{lemma}\label{lem:summation}
    Let $\gamma, L > 0$, and suppose $f \in L^1_{t, \loc} ([0, \infty) \to \R^+)$, then for all $t_0 \geq 0$, we have
        \begin{equation}
            \int_{t_0}^\infty e^{- \gamma t} f(t) \, dt
                \lesssim \frac{1}{\gamma L} e^{-\gamma t_0} \sup_{\substack{I \subseteq [0, \infty) \\ |I| = L}} \| f \|_{L^1_t (I)}.
        \end{equation}
\end{lemma}

\begin{proof}
    We interpolate between the uniform boundedness on unit-time scales and decay, dividing $[0, \infty)$ into sub-intervals $I_n := t_0 + [nL, (n + 1) L]$ and estimating 
        \begin{align*}
            \int_{t_0}^\infty e^{- \gamma t} f(t) \, dt 
                &\leq e^{-\gamma t_0} \sum_{n = 0}^\infty e^{- n \gamma L} \int_{I_n} f(t) \, dt 
                \lesssim \frac{1}{\gamma L} e^{-\gamma t_0} \sup_{\substack{I \subseteq [0, \infty) \\ |I| = L}} \| f \|_{L^1_t (I)},
        \end{align*}
    as desired. 
\end{proof}

\begin{proof}[Proof of $L^2$-convergence to $\cM^N$]
    Our goals here are two-fold; first, we claim the $L^2$-difference bound
        \[
            \| (A, \phi)(t_2) - (A, \phi)(t_1) \|_{L^2_x}^2
                    \lesssim_N e^{-\gamma t} \| \Bogomolnyi^{\mathrm{in}} \|_{L^2_x}^2 .
        \]
    This would imply that the flow $\{ (A(t), \phi(t)) \}_t$ is Cauchy with respect to the $L^2$-metric as $t \to \infty$, from which we may extract a limiting configuration $(A^\infty, \phi^\infty)$. By decay of the Bogomoln'yi tension field \eqref{bogomolnyi-decay}, this limit is a weak solution to the Bogomoln'yi equations \eqref{bogomolnyi} -- our second order of business then would be to show that the topological degree is $N$, and thereby residing in the moduli space $(A^\infty, \phi^\infty) \in \cM^N$.

    We accomplish both goals in concert by estimating the $L^1_t L^2_x$-norm of the time-derivatives. Performing this estimate at the level of the self-dual formulation of the equation \eqref{schematic-SD-AHG},
    \begin{align*}
        \int_{t_1}^{t_2} \| \partial_t (A, \phi) \|_{L^2_x} \, dt
            &\lesssim \int_{t_1}^{t_2} \| \phi \|_{L^\infty_x} \| \Bogomolnyi \|_{L^2_x} + \| \sfD \Bogomolnyi \|_{L^2_x} \, dt \\
            &\lesssim \Bigl( \int_{t_1}^{t_2} e^{-\frac12 \gamma t} \| \phi \|_{L^\infty_x} \, dt\Bigr) \Bigl(\sup_{t \in [t_1, t_2]} e^{\gamma t} \| \Bogomolnyi(t) \|_{L^2_x}^2 \Bigr)^\frac12\\
                &\qquad  +  \Bigl( \int_{t_1}^{t_2}   e^{-\gamma t} \, dt  \Bigr)^\frac12 \Bigl( \int_{t_1}^{t_2} e^{\gamma t} \| \sfD \Bogomolnyi \|_{L^2_x}^2 \, dt\Bigr)^{\frac12} 
            \lesssim_N e^{-\frac\gamma2 t_1} \| \Bogomolnyi^{\mathrm{in}} \|_{L^2_x},
    \end{align*}
    using the summability bound from Lemma \ref{lem:summation} and the unit-time scale $L^1_t L^\infty_x$-bound from Proposition \ref{cor:Lp-Linfty}, and applying Bogomoln'yi tension field decay \eqref{bogomolnyi-decay}. Then the $L^2$-difference bound is immediate from the fundamental theorem of calculus, while the preservation of topological degree $N$ in the limit as $t \to \infty$ follows from Proposition \ref{lem:vorticity-II}. 
\end{proof}

\begin{proof}[Proof of $\dot H^1$-convergence of $A$]
    To estimate the differences of the gradient of the gauge potential along the gradient flow, it suffices by the Helmholtz decomposition to estimate that the differences of the divergence and the differences of the curl, i.e. we are reduced to estimating the right-hand side of
        \[
            \| \nabla A (t) - \nabla A^\infty \|_{L^2_x} 
                \lesssim \| \partial^\ell A_\ell (t) - \partial^\ell A_\ell^\infty \|_{L^2_x} + \| F_{12} (t) - F_{12}^\infty \|_{L^2_x}.
        \]
    
    For the divergence, a straightforward calculation using the equation \eqref{AHG} gives
        \[
            \partial^\ell F_{t \ell} = \Im(\overline \phi \bfD_t \phi).
        \] 
    Integrating this expression on $[t_1, t_2]$ in temporal gauge, the left-hand side gives the differences of the divergence by the fundamental theorem of calculus, 
        \begin{align*}
            \partial^\ell A_\ell (t) - \partial^\ell A_\ell^\infty 
                = \int_{t}^{\infty} \Im(\overline \phi \bfD_t \phi) \, dt. 
        \end{align*}
    
    For the curl, we can recast the difference of the magnetic fields in terms of the charge density and the Gauss tension fields by definition, and then rewrite the difference of the charge densities using the fundamental theorem of calculus, 
        \begin{align*}
            F_{12} (t) - F_{12}^\infty 
                &= \bigl( \Gauss(t) - \Gauss^\infty  \bigr) +  \tfrac12\bigl( |\phi(t)|^2 - |\phi^\infty|^2 \bigr)  \\
                &=  \Gauss(t) - \int_{t}^{\infty} \Re(\overline \phi \bfD_t \phi) \, dt'. 
        \end{align*}
    We can immediately handle the Gauss tension fields via the decay estimate \eqref{bogomolnyi-decay}, so it remains to handle the integral of $\phi \cdot \bfD_t \phi$ in both expessions above. We rewrite it using the equation \eqref{schematic-B} and estimate the $L^2_x$-norm thereof, 
        \begin{equation}\label{eq:phi-Dt}
        \begin{split}
            \int_{t}^{\infty} \| \phi \cdot \bfD_t \phi \|_{L^2_x} \, dt' 
                &\lesssim \int_{t}^\infty \| \phi \|_{L^\infty_x} \| \sfD \Bogomolnyi \|_{L^2_x} + \| \phi \|_{L^\infty_x}^2 \| \Bogomolnyi \|_{L^2_x} \, dt' \\
                &\lesssim  \Bigl( \int_{t}^{\infty}  e^{-\gamma t'} \| \phi \|_{L^\infty_x}^2 \, dt  \Bigr)^\frac12 \Bigl( \int_{t}^{\infty} e^{\gamma t'} \| \sfD \Bogomolnyi (t')\|_{L^2_x}^2 \, dt';\Bigr)^{\frac12} \\
                    &\qquad + \Bigl( \int_{t}^\infty e^{-\frac12 \gamma t'} \| \phi \|_{L^\infty_x}^2 \, dt'\Bigr) \Bigl(\sup_{t \in [t, \infty]} e^{\gamma t'} \| \Bogomolnyi(t') \|_{L^2_x}^2 \Bigr)^\frac12 
            \lesssim_N e^{-\frac\gamma2 t} \| \Bogomolnyi^{\mathrm{in}} \|_{L^2_x},
        \end{split}
        \end{equation}
     using the summability bound from Lemma \ref{lem:summation} and the unit-time scale $L^2_t L^\infty_x$-bounds from Proposition \ref{cor:Lp-Linfty}, and applying Bogomoln'yi tension field decay \eqref{bogomolnyi-decay}. 
\end{proof}

\subsection{$\dot H^1$-convergence of scalar field} \label{sec:H1-phi}

Our last order of business is to upgrade the $(H^1 \times L^2)$-convergence \eqref{main-converge} to a $\dot H^1$-convergence bound for the scalar field \eqref{main-converge-2} under the additional assumption that the initial data obeys $A^{\mathrm{in}} \in L^p (\R^2)$ for $2 < p < \infty$. To better leverage the smoothing effects of the gradient flow and nail down the size of the magnetic potential, it is convenient to pass to Coulomb gauge \eqref{coulomb} at $t = \infty$ and integrate backwards-in-time from the limiting configuration. To accomplish this first step, set
    \[
        \chi 
            := - \Delta^{-1} \partial^\ell A_\ell^\infty,
    \]
then the gauge-transformed configuration $(A + d \chi, e^{i \chi}\phi)$ satisfies the Coulomb gauge at $t = \infty$. To obtain a bound in the original gauge, we need to estimate effects of the gauge transformation, 

\begin{lemma}[$H^1$-norm of gauge-transformed field]\label{lem:transform}
    Let $\chi \in L^1_\loc (\R^2 \to \R)$ be a gauge-transformation with regularity $\nabla \chi \in L^p (\R^2)$ for $2 < p \leq \infty$, and suppose $u \in H^1 (\R^2)$. Then 
        \begin{equation}
            \| u e^{i \chi}\|_{H^1_x}
                \lesssim \Bigl( 1 + \| \nabla \chi \|_{L^p_x} \Bigr) \| u \|_{H^1_x}.
        \end{equation}
\end{lemma}

\begin{proof}
    The $L^2$-norm is gauge-invariant, so it remains to calculate the derivative of the gauge-transformed field. By the product rule, 
        \[
            \nabla \bigl( u e^{i \chi} \bigr)
			= e^{i \chi} \bigl( \nabla u + i u \nabla \chi \bigr).
        \]
    We bound this by 
        \begin{align*}
            \bigl\| \nabla \bigl( u e^{i \chi} \bigr) \bigr\|_{L^2_x}
                &\leq \| \nabla u \|_{L^2_x} +  \| u \|_{L^q_x} \bigl\| \nabla \chi \bigr\|_{L^p_x} \\
                &\lesssim \Bigl( 1 + \| \nabla \chi \|_{L^p_x} \Bigr) \| u \|_{H^1_x},
        \end{align*}
    using the triangle inequality and H\"older's inequality with $\tfrac12 = \tfrac1p + \tfrac1q$ in the first line, and and Sobolev embedding in the second line. This completes the proof. 
\end{proof}

By boundedness of the Riesz transforms, the difference bound \eqref{main-converge}, and Sobolev embedding, 
    \begin{align*}
        \| \nabla \chi \|_{L^p_x} 
            &\lesssim \bigl\| \nabla\Delta^{-1} \partial^\ell (A_\ell^\infty - A_\ell^{\mathrm{in}}) \bigr\|_{L^p} + \bigl\| \nabla\Delta^{-1} \partial^\ell  A^{\mathrm{in}}_\ell \bigr\|_{L^p_x} \lesssim_N 1 + \bigl\| A^{\mathrm{in}} \bigr\|_{L^p} ,
    \end{align*}
so to prove \eqref{main-converge-2}, it suffices from \eqref{main-converge} and Lemma \ref{lem:transform} to prove 
    \begin{equation}
        \bigl\| \nabla \phi(t) - \nabla \phi^\infty \bigr\|_{L^2}
            \lesssim_N e^{-\frac\gamma2 t}\bigl\| \Bogomolnyi^{\mathrm{in}} \bigr\|_{L^2_x} 
    \end{equation}
for $(A^\infty, \phi^\infty)$ satisfying the Coulomb gauge \eqref{coulomb}. 

By ellipticity, it suffices to estimate the Cauchy-Riemann derivative of $\phi$. That is, we are reduced to estimating the right-hand side of
    \[
        \| \nabla \phi (t) - \nabla \phi^\infty \|_{L^2_x} 
            \lesssim \| \partial_{\overline z} \phi (t) - \partial_{\overline z} \phi^\infty \|_{L^2_x}.
    \]
    Rewriting the Cauchy-Riemann derivatives in terms of in terms of their covariant counterparts, noting that $(A^\infty, \phi^\infty)$ solves \eqref{bogomolnyi} and rewriting the lower-order terms via the fundamental theorem of calculus, 
        \begin{align*}
            \partial_{\overline z} \phi (t) - \partial_{\overline z} \phi^\infty 
                &= \bigl( \bfD_{\overline z} \phi (t) - \bfD_{\overline z} \phi^\infty \bigr) - i \bigl( A_{\overline z}^\infty \phi^\infty - A_{\overline z} (t) \phi(t) \bigr)\\
                &= \bfD_{\overline z} \phi(t) - i \int_{t}^\infty A_{\overline z} (t') \cdot \partial_t \phi(t') + \partial_t A_{\overline z} (t') \cdot \phi(t') \,  dt'.
        \end{align*}
    We can immediately estimate the first term on the right via the Bogomoln'yi tension field decay \eqref{bogomolnyi-decay}, so it remains to handle the integral.  The term $\partial_t A_{\overline z} \cdot \phi$ may be handled by \eqref{phi-Dt} from the previous proof, so we focus on the term $A_{\overline z} \cdot \partial_t \phi$. As a matter of fact, we may argue fairly similarly with this term, rewriting using the equation \eqref{schematic-B}, and then amply applying Cauchy-Schwarz, 
        \begin{align*}
            \int_t^\infty \bigl\| A_{\overline z} \cdot \partial_t \phi  \bigr\|_{L^2_x} \, dt 
                &\lesssim \int_t^\infty  \| A \|_{L^\infty_x} \bigl( \| \sfD \Bogomolnyi \|_{L^2_x} + \| \phi \|_{L^\infty_x} \| \Bogomolnyi \|_{L^2_x} \bigr) \, dt \\
                &\lesssim \Bigl( \int_t^\infty e^{-\gamma t'} \| A \|_{L^\infty_x}^2 \, dt \Bigr)^{\frac12} \Bigl( \int_t^\infty e^{\gamma t'} \| \sfD \Bogomolnyi \|_{L^2_x}^2 \, dt' \Bigr)^{\frac12} \\
                    &\qquad + \Bigl( \int_t^\infty e^{-\frac12 \gamma t'} \bigl( \| A \|_{L^\infty_x}^2 + \| \phi\|_{L^\infty_x}^2 \bigr)\, dt'\Bigr) \Bigl( \sup_{t' \in [t_1, t_2]} e^{\gamma t} \| \Bogomolnyi(t') \|_{L^2_x}^2 \Bigr)^{\frac12}\\
                &\lesssim_N e^{-\frac\gamma2 t} \| \Bogomolnyi^{\mathrm{in}} \|_{L^2_x},
        \end{align*}
    using the unit-time scale $L^2_t L^\infty_x$-bounds for $\phi$ from Propositions \ref{cor:Lp-Linfty}, a unit-time scale $L^2_t L^\infty_x$-bound for $A$ which we will justify below, then summing via Lemma \ref{lem:summation} and leveraging the decay \eqref{bogomolnyi-decay} to conclude. 

    It remains to justify 
         \begin{equation}\label{eq:claim-A}
            \| A \|_{L^2_t L^\infty_x ([t, t + t_*])} 
                \lesssim_N 1.
        \end{equation}
    Rewriting the magnetic potential using the fundamental theorem of calculus,
        \begin{align*}
            A_{\overline z} (t)
                = A_{\overline z}^\infty - \int_t^\infty \partial_t A_{\overline z} (t') \, dt'.
        \end{align*}
    Writing $A^\infty$ in terms of its curvature via the Biot-Savart law, and using the boundedness \eqref{vortex-Linfty} and quantisation \eqref{quantised} of the magnetic field, we can bound 
        \[
            \|A^\infty \|_{L^\infty_x} 
                \lesssim \| F_{12}^\infty \|_{L^1_x} + \| F_{12}^\infty \|_{L^\infty_x} \lesssim 1 + N. 
        \]
    For the contribution of the integral, we claim that in fact it is $L^\infty_{t, x}$ on $[t_*, \infty)$, and is $L^p_t L^\infty_x$ on $[0, t_*]$ for every $1 \leq p < \infty$. 
    
    For $t \in [t_*, \infty)$, we rewrite the integrand using the equation \eqref{schematic-SD-AHG}, and partition into sub-intervals $I$ of length $\tfrac12 t_*$. To estimate on each sub-interval, we run the gradient flow starting from $I - \tfrac12 t_*$ and use the parabolic smoothing alotted by Proposition \ref{prop:smooth-II},
        \begin{align*}
            \int_{t}^\infty \| \partial_t A(t') \|_{L^\infty_x} \, dt'
                &\lesssim  \sum_{n = 1}^\infty \int_{[\frac{n}{2} t_*, \frac{n + 1}{2} t_*]} \| \phi(t') \|_{L^\infty_x} \| \Bogomolnyi (t') \|_{L^\infty_x} + \| \sfD \Bogomolnyi(t') \|_{L^\infty_x} \, dt'\\ 
                &\lesssim \sum_{n = 1}^\infty \| \phi \|_{L^4_t L^\infty_x ([\frac{n}{2} t_*, \frac{n + 1}{2}])} \| |t"|^{-\frac12} \|_{L^{4/3}_{t"} ([\frac12 t_*, t_*])}  \Bigl( \sup_{t" \in [\frac12 t_* , t_*]} |t"|^{\frac12} \| \Bogomolnyi (\tfrac{n - 1}{2} t_* + t") \|_{L^\infty_x} \Bigr)\\
                    &\qquad + \Bigl( \int_{[\frac12 t_*, t_*]} |t"|^{-1} \, dt" \Bigr) \Bigl( \sup_{t" \in [\frac12 t_*, t_*]} |t"|  \| \sfD \Bogomolnyi(\tfrac{n - 1}{2} t_* + t") \|_{L^\infty_x} \Bigr) \\
                &\lesssim_N \sum_{n = 1}^\infty \| \Bogomolnyi (\tfrac{n - 1}{2} t_*) \|_{L^2_x} \\
                &\lesssim_N \sum_{n = 1}^\infty e^{- \gamma \frac{n - 1}{2} t_*} \lesssim_N 1, 
        \end{align*}
using the smoothing bound \eqref{B-smoothing} and the unit-time scale $L^4_t L^\infty_x$-bound for the scalar field \eqref{Lp-Linfty} in the third line, and summing via the decay \eqref{bogomolnyi-decay}. 

For $t \in [0, t_*]$, we run the gradient flow starting from $t = 0$, and a similar argument yields
    \begin{align*}
        \int_t^{t_*} \| \partial_t A (t') \|_{L^\infty_x} \, dt' 
            &\lesssim \int_t^{t_*} \| \phi (t') \|_{L^\infty_x} \| \Bogomolnyi (t') \|_{L^\infty_x} + \| \sfD \Bogomolnyi (t') \|_{L^\infty_x} \, dt'\\
            &\lesssim \| \phi \|_{L^4_{t'} L^\infty_x ([0, t_*])} \| |t'|^{-\frac12} \|_{L^{4/3}_{t'} ([0, t_*])} \Bigl( \sup_{t' \in [0, t_*]} |t"|^{\frac12} \| \Bogomolnyi (t') \|_{L^\infty_x} \Bigr)\\
                &\qquad + \Bigl( \int_t^{t_*} |t'|^{-1} \, dt" \Bigr) \Bigl( \sup_{t' \in [0, t_*]} |t'|  \| \sfD \Bogomolnyi(t') \|_{L^\infty_x} \Bigr)\\
            &\lesssim_N 1 + \log(t_*/t).
    \end{align*}
This is clearly $L^2_t$-integrable on $[0, t_*]$, so we conclude \eqref{claim-A}.

\appendix
\section{Well-posedness theory}\label{sec:gwp}

In this appendix, we resolve a subtle gap in the literature concerning the global well-posedness of the abelian Higgs gradient flow (for arbitrary coupling constant $\lambda > 0$),
	\begin{equation}\label{eq:app-AHG}\tag{AHG}
		\begin{split}
			\bfD_t \phi - \bfD^j \bfD_j \phi
				&= \tfrac\lambda2(1 - |\phi|^2) \phi, \\
			F_{tj}
				&= \Im(\overline \phi \bfD_j \phi) + \partial^\ell F_{\ell j}, 
		\end{split}
	\end{equation}
with initial data 
	\[
		(A, \phi)_{|t = 0}
			= (A^{\mathrm{in}}, \phi^{\mathrm{in}}),
	\]
in the \textit{energy space}, 
	\[
		\frE 
			:= \bigl\{ (A, \phi) \in H^1_{\loc} (\R^2) : \cE[A, \phi] < \infty  \text{ and } A \in L^p (\R^2) \text{ for some $2 < p < \infty$} \bigr\}.
	\]
Demoulini-Stuart \cite[Section 4]{DemouliniStuart1997} prove global well-posedness for \eqref{app-AHG} on the scale of $C^{k, \alpha}$-spaces. The global well-posedness in the affine space $(A^{\lambda, N}, \phi^{\lambda, N}) + H^1 (\R^2)$ linearised about a vortex is taken as an assumption in Gustafson \cite{Gustafson2002}. Neither, at least \textit{a priori}, explore a dense\footnote{of course, the precise meaning of ``dense'' here is left unclear.} subset of the energy space, and so are unsatisfactory for the generality we are interested in. On the other hand, there are rich subtleties stemming from the non-linear character of the phase space which would require a whole article to explore in detail. For this reason, we provide a reasonably self-contained but abridged account of the well-posedness theory for \eqref{AHG},

\begin{theorem}[Global well-posedness of abelian Higgs gradient flow]
	For any finite-energy configuration $ (A^{\mathrm{in}}, \phi^{\mathrm{in}}) \in \frE$, there exists a finite-energy configuration $(A, \phi)$ on $[0, \infty) \times \R^2$ solving the Cauchy problem for the abelian Higgs gradient flow \eqref{app-AHG} in temporal gauge $A_t = 0$ with initial data $(A, \phi)_{|t = 0} = (A^{\mathrm{in}}, \phi^{\mathrm{in}})$. Furthermore, this solution satisfies the energy identity,
	\begin{equation}\label{eq:energy-identity}
		\frac{d}{dt} \cE[A, \phi] 
			= - \int_{\R^2} | \bfD_t \phi|^2 + |F_{t1}|^2 + |F_{t2}|^2 \, dx.
	\end{equation}
\end{theorem}

Rather than directly work in temporal gauge, under which the equation for $A$ is an ODE, and formulate well-posedness in the energy space, which in itself is a non-linear phase space, we will content ourselves with (\texttt{i}) placing the initial data in Coulomb gauge, (\texttt{ii}) working in the \textit{DeTurck gauge}, ``revealing the parabolicity'' of the equation for $A$, (\texttt{iii}) linearising the equation \eqref{app-AHG} around a smooth background and (\texttt{iv})  proving well-posedness for equation satisfied by the perturbation in $H^1 (\R^2 \to \R^2 \times \C)$. Passing to this problem loses no generality, as, with regards to (\texttt{i}), we can convert any initial data $(A^{\mathrm{in}}, \phi^{\mathrm{in}})$ to Coulomb gauge via 
	\[
		(A^{\mathrm{in}}, \phi^{\mathrm{in}}) 
			\mapsto (A^{\mathrm{in}} + d \chi, \phi^{\mathrm{in}} e^{i \chi}), \qquad \text{where } \chi := - \Delta^{-1} \partial^\ell A_\ell^{\mathrm{in}}, 
	\]
and, with regards to (\texttt{ii}), any configuration $(A, \phi)$ to temporal gauge via 
	\[
		(A, \phi) \mapsto (A + d \chi, \phi e^{i \chi}), \qquad \text{where }\chi(t)
			:= - \int_0^t A_t (t') \, dt'.
	\]
The third point (\texttt{iii}) is more technical, and requires a bit of functional analysis for the energy space, which is contained in Section \ref{sec:energy}. This will also lay the groundwork for the energy estimates for (\texttt{iv}) in Section \ref{sec:lwp}, which will allow us to conclude global well-posedness. 

\subsection{Energy space}\label{sec:energy}

Define the \textit{space of smooth configurations} by
	\[
		\frE^\infty 
			:= \bigl\{ (A, \phi) \in \frE : \| \bfD^{(n)} \phi \|_{L^2} + \| \nabla^{(n - 1)} F_{12} \|_{L^2} < \infty \text{ for all $n \in \N$}  \bigr\}.
	\]
We want to decompose the energy space into $\frE = \frE^\infty + H^1 (\R^2)$, and show that the energy functional is continuous with respect to this decomposition. 

\begin{lemma}[Regularity of smooth configurations]\label{lem:Linfty-config}
	Let $(A, \phi) \in \frE^\infty$ be a smooth configuration in Coulomb gauge. Then $(A, \phi) \in L^\infty (\R^2)$ and $A \in L^4 (\R^2)$. 
\end{lemma}

\begin{proof}
	The bound for the scalar field follows immediately from \eqref{inefficientLinfty}. For the magnetic potential, we write
		\begin{align*}
            A 
                &= - \epsilon_{jk} \partial^k \Delta^{-1} \omega(A, \phi) - \mathbb P^{\mathrm{df}} (\overline \phi \bfD \phi), \\
            A 
                &= - \epsilon_{jk} \partial^k \Delta^{-1} F_{12} ,
        \end{align*}
	where $\mathbb P^{\mathrm{df}}$ is the projection to divergence-free vector fields. The latter is the standard Biot-Savart law, and is favourable at top-order. The former is a modified Biot-Savart law arising from the expression
		 \[
            A  
                = \mathbb P^{\mathrm{df}} (A + \Im(\overline \phi \bfD \phi)) - \mathbb P^{\mathrm{df}} (\overline \phi \bfD \phi),
        \]
	and is favourable at bottom-order due to the better spatial localisation of the vorticity $\omega(A, \phi) \in L^1 (\R^2)$, recalling the discussion from Lemma \ref{lem:vorticity-I}, compared to the magnetic field $F_{12} \in L^2 (\R^2)$. With these two expressions at our disposal, we can conclude
		\begin{align*}
			\|  A \|_{L^\infty_x} 
				&\lesssim \| A \|_{L^4_x} + \| \nabla A \|_{L^4_x} \\
				&\lesssim \|  \omega (A, \phi)\|_{L^{4/3}_x} + \| \phi \|_{L^\infty_x} \| \underline \bfD \underline \phi \|_{L^4_x} + \| \underline F_{12} \|_{L^4_x}  \lesssim_\cE 1,
		\end{align*}
	using \eqref{inefficientLinfty} and Gagliardo-Nirenberg appropriately. 
\end{proof}

\begin{remark}
	One can carry out the same argument for higher-order derivatives to show that smooth configurations satisfying the Coulomb gauge are in fact smooth in the usual sense, that is, $\frE^\infty_{\mathrm{df}} \subseteq C^\infty (\R^2)$. 
\end{remark}

\begin{proposition}[Decomposition of the energy space]\label{prop:decomp}
	Let $(A^{\mathrm{in}}, \phi^{\mathrm{in}}) \in \mathfrak E$ be a finite-energy configuration in Coulomb gauge. Then 
		\begin{equation}
			(A^{\mathrm{in}}, \phi^{\mathrm{in}})
				= (\underline A, \underline \phi) + (a, u),
		\end{equation}
	where $(\underline A, \underline \phi) \in \mathfrak E^\infty$ is a smooth configuration in Coulomb gauge and $(a, u) \in H^1 (\R^2 \to \R^2 \times \C)$. 
\end{proposition}

\begin{proof}
	We introduce the \textit{covariant linear heat equation},
		\begin{equation}\label{eq:linear}
		\begin{split}
			(\bfD_s - \bfD^\ell \bfD_\ell) \phi 
				&= 0, \\
			F_{s j}
				&= \partial^\ell F_{\ell j},
		\end{split}
		\end{equation}
	where $(A, \phi)$ is a configuration on $\R^2 \times [0, \infty)$. Imposing the \textit{caloric gauge} $A_s = 0$ and taking the divergence of the second equation, we obtain $\partial_s \partial^\ell A_\ell = 0$. Thus, the Coulomb gauge \eqref{coulomb} is propagated along the flow of \eqref{linear} in caloric gauge. Imposing initial data 
		\[
			(A, \phi)_{|s = 0} 
				:= (A^{\mathrm{in}}, \phi^{\mathrm{in}})
		\]
	for \eqref{linear} under caloric gauge, the DeTurck gauge $A_s = \partial^\ell A_\ell$ is trivially satisfied, so the second equation reduces to the linear heat equation for $A$, and the Cauchy problem is well-posed. One can show covariant parabolic smoothing estimates similar to those in Section \ref{sec:smooth} up to the time-scale $s \ll (1 + \cE)^{-1}$; we leave the details to the reader. Set then $\underline s:= c(1 + \cE)^{-1}$ for some small $c \ll 1$, and define the regularised configuration by 
		\[
			 (\underline A, \underline \phi)
				:= (A, \phi)_{|s = \underline s}.
		\]
	An immediate consequence of the covariant smoothing estimates is that $(\underline A, \underline \phi) \in \frE^\infty$. It remains to show that the difference satisfies $(\underline A, \underline \phi) - (A^{\mathrm{in}}, \phi^{\mathrm{in}}) \in H^1 (\R^2)$. 

	At bottom-order, we use the fundamental theorem of calculus, the equation \eqref{linear}, and the covariant smoothing estimates to bound 
		\begin{align*}
			\bigl\|(\underline A, \underline \phi) - (A^{\mathrm{in}}, \phi^{\mathrm{in}}) \bigr\|_{L^2_x} 
				&\lesssim \int_0^{\underline s} \| (\partial_s A , \partial_s \phi) \|_{L^2_x} \, ds\\
				&\lesssim \Bigl(\int_0^{\underline s} s^{-\frac12} \, ds\Bigr) \Bigl( \bigl\| F_{12}^{\mathrm{in}} \bigr\|_{L^2_x} + \| \bfD \phi^{\mathrm{in}} \|_{L^2_x} \Bigr)\, ds  \lesssim 1.
		\end{align*}

	The magnetic potentials are divergence-free, so their derivatives may be estimated by the curls, which has monotonically decreasing $L^2$-norm under the linear heat flow. Thus, 
		\begin{align*}
			\bigl\| \nabla \underline A - \nabla A^{\mathrm{in}} \bigr\|_{L^2_x}
				&\lesssim \bigl\| F_{12}^{\mathrm{in}} \bigr\|_{L^2_x} \lesssim_\cE 1.
		\end{align*}

	For the derivatives of the scalar-fields, it is convenient to prove an $L^\infty$-bound for $(A, \phi)$ along the flow. By Lemma \ref{lem:Linfty-config}, this immediately holds for the regularised configuration, so, integrating from $s = \underline s$ using the fundamental theorem of calculus,
		\begin{align*}
			\| (A,\phi) (s) \|_{L^\infty_x}
				&\lesssim \bigl\| (\underline A, \underline \phi) \bigr\|_{L^\infty_x} + \int_s^{\underline s} \|(\partial_s A, \partial_s \phi) \|_{L^\infty_x} \, ds' \\
				&\lesssim_\cE 1 + \Bigl( \int_s^{\underline s} |s'|^{-1} \, ds'\Bigr)\\
				&\lesssim 1 + \log(\underline s/s),
		\end{align*}
	controlling the high-frequency contributions using the equation \eqref{linear} and the covariant parabolic smoothing. we write the difference of the derivatives in terms of that of their covariant counterparts, and rewrite the lower-order terms via the fundamental theorem of calculus, 
		\begin{align*}
			\bigl\| \nabla \underline \phi - \nabla \phi^{\mathrm{in}} \bigr\|_{L^2_x} 
				&\lesssim \bigl\| \underline \bfD \underline \phi \bigr\|_{L^2_x} + \bigl\| \bfD \phi^{\mathrm{in}} \bigr\|_{L^2_x} + \int_0^{\underline s} \| \phi \|_{L^\infty_x} \| \partial_s A \|_{L^2_x} + \| A \|_{L^\infty_x} \| \partial_s \phi \|_{L^2_x} \, ds \\
				&\lesssim_\cE 1 + \int_0^{\underline s} (1 + \log(\underline s/s)) s^{-\frac12} \, ds \\
				&\lesssim_\cE 1. 
		\end{align*}
	This completes the proof. 
\end{proof}

\begin{proposition}[Continuity of the energy functional]\label{prop:lipschitz}
	Let $(\underline A, \underline \phi) \in \frE^\infty$ be a smooth configuration in Coulomb gauge. Then the map 
		\begin{align*}
			H^1 (\R^2) 
				&\longrightarrow L^2 (\R^2) \\
			(a, u) 
				&\longrightarrow \bigl( \tfrac12(1 - |\underline \phi + a|^2), \epsilon_{jk} \partial^j (\underline A_k + a_k), (\underline \bfD - i a)(\underline \phi + u) \bigr)
		\end{align*}
	is well-defined and locally Lipschitz continuous. Furthermore, the energy 
		\begin{align*}
			H^1 (\R^2 \to \R^2 \times \C) 
				&\longrightarrow [0, \infty), \\
			(a, u)
				&\longmapsto \cE[\underline A + a, \underline \phi + u]
		\end{align*}
	is well-defined and locally Lipschitz continuous. 
\end{proposition}

\begin{proof}
	Let $(a, u), (b, v) \in H^1(\R^2 \to \R^2 \times \C)$, then by the triangle inequality, Sobolev embedding, and the $L^\infty$-bounds from Lemma \ref{lem:Linfty-config}, we can write 
		\begin{align*}
			 \bigl\||\underline \phi + u|^2 - |\underline \phi - v|^2 \bigr\|_{L^2} 
				&\lesssim \| \underline \phi \|_{L^\infty} \| u - v \|_{L^2_x} + \| |u|^2 - |v|^2 \|_{L^2}\\
				&\lesssim \| \underline \phi \|_{L^\infty} \Bigl( \|u\|_{H^1} + \|v \|_{H^1} \Bigr) \| u - v\|_{H^1}
		\end{align*}
	and 
		\begin{align*}
			\bigl\| (\underline \bfD - i a) (\underline\phi + u) - (\underline \bfD - i b) (\underline\phi + v)  \bigr\|_{L^2}
				&\lesssim \| \underline \bfD u - \underline \bfD v \|_{L^2} + \| a \underline \phi - b \underline \phi \|_{L^2} + \| au - bv \|_{L^2}\\
				&\lesssim \Bigl( 1 + \| (\underline A, \underline \phi)\|_{L^\infty} + \| (a, u) \|_{H^1} + \| (b, v)\|_{H^1} \Bigr) \\
					&\qquad \times\| (a, u) - (b, v) \|_{H^1},
		\end{align*}
	and 
		\[
			\bigl\| \epsilon_{jk} \partial^j (\underline A_k + a_k) - \epsilon^{jk} \partial^j (\underline A_k + b_k) \bigr\|_{L^2}
				\leq \|\nabla a - \nabla b \|_{L^2}.
		\]
	The continuity of the energy functional follows from the triangle inequality. 
\end{proof}

\subsection{Proof of global well-posedness}
\label{sec:lwp}
As done in Demoulini-Stuart \cite[Section 4]{DemouliniStuart1997}, for the initial data problem we find it convenient to pass to the \textit{DeTurck gauge}, 
	\begin{equation}\label{eq:DT}\tag{DT}
		A_t 
			= \partial^\ell A_\ell. 
	\end{equation}
This renders the equation \eqref{app-AHG} a non-linear heat equation for $(A, \phi)$,

\begin{lemma}[Abelian Higgs gradient flow in DeTurck gauge]
	Let $(A, \phi)$ be a configuration on $[0, T] \times \R^2$ solving the abelian Higgs gradient flow \eqref{app-AHG} in DeTurck gauge \eqref{DT}. Then the equation reduces to 
		\begin{equation}\label{eq:lambda-AHG}
		\begin{split}
			(\partial_t - \Delta) \phi
				&= - 2i A^\ell \bfD_\ell \phi + A^\ell A_\ell \phi + \tfrac\lambda2(1 - |\phi|^2) \phi, \\
			(\partial_t - \Delta) A_j
				&= \Im(\overline \phi \bfD_j \phi) .
		\end{split}
		\end{equation}
	\end{lemma}
We would like to regard this as a perturbation of the linear heat equation. With this viewpoint in mind, recall the standard linear parabolic estimate, 

\begin{lemma}[Linear energy estimate]
	Let $u: [0, T] \times \R^2 \to \C$ be a sufficiently regular solution to the linear heat equation 
		\begin{equation}
			(\partial_t - \Delta) u 
				= \cF + \nabla \cG. 
		\end{equation}
	Then, for each non-negative integer $k \in \N_0$, we have 
		\begin{equation}\label{eq:energy-estimate}
			\| u \|_{L^\infty_t H^k_x} + \| \nabla u \|_{L^2_t H^k_x} 
				\lesssim \| u^{\mathrm{in}} \|_{L^2_x} + \| \cF \|_{L^1_t H^k_x} + \| \cG \|_{L^2_t H^k_x}. 
		\end{equation}
\end{lemma}

\begin{proof}
	Commute the multiplier $\langle \nabla \rangle^k$ through the equation, and then test the equation against $\langle \nabla \rangle^k u$. Integrating-by-parts appropriately on the left, moving the derivative onto the multiplier on the right, and adequately applying Cauchy-Schwartz gives the result. 
\end{proof}

By Proposition \ref{prop:decomp}, to construct a solution to \eqref{lambda-AHG} with energy-class initial data, it suffices to work with $H^1$-perturbations of a smooth background $(\underline A, \underline \phi) \in \frE^\infty$. This allows us to work with standard linear function spaces. Global well-posedness follows from a standard Picard iteration argument and energy monotonicity, 

\begin{theorem}[Global well-posedness of abelian Higgs gradient flow in DeTurck gauge]
	Let $(\underline A, \underline \phi) \in \frE^\infty$ be a smooth configuration in Coulomb gauge \eqref{coulomb}. Then the abelian Higgs gradient flow in DeTurck gauge \eqref{lambda-AHG} is globally well-posed in $(\underline A, \underline \phi) + H^1 (\R^2 \to \R^2 \times \C)$ and satisfies the energy identity \eqref{energy-identity}.
\end{theorem}

\begin{proof}[Proof of local existence]
	Set 
		\[
			(a, u) 
				:= (A, \phi) - (\underline A, \phi), 
		\]
	then by Duhamel's principle, a solution to the Cauchy problem for \eqref{lambda-AHG} may be recast as a fixed point of the non-linear mapping
		\begin{align*}
			\Phi 
				: \bigl(L^\infty_t H^1_x \cap L^2_t \dot H^2_x\bigr) ([0, T] \times \R^2)
					&\longrightarrow \bigl(L^\infty_t H^1 \cap L^2_t H^2_x\bigr) ([0, T] \times \R^2) \\
				(a, u) 
					&\longmapsto e^{t \Delta} (a^{\mathrm{in}}, u^{\mathrm{in}}) + (\partial_t - \Delta)^{-1} \Bigl( \Delta(\underline A, \underline \phi) + \cN (A, \phi) \Bigr), 
		\end{align*}
	where the ``non-linear''\footnote{obviously, upon expanding in terms of the perturbation, these will contain zero-th order terms and linear terms.} terms are given by 
		\begin{align*}
			\cN(A, \phi)
				&:= \bigl( \Im(\overline \phi \bfD \phi), - 2i A^\ell \bfD_\ell \phi + A^\ell A_\ell \phi + \tfrac\lambda2 (1 - |\phi|^2) \phi  \bigr).
		\end{align*}
	We argue by Picard iteration, aiming to show that, for initial data of size $R := \| (a^{\mathrm{in}}, u^{\mathrm{in}}) \|_{H^1_x}$, there exists a small time $T \ll_R 1$ for which this map is contractive on a ball of radius $2R > 0$. That is, 
		\begin{align}
			\| \Phi (a, u) \|_{L^\infty_t H^1_x \, \cap \, L^2_t \dot H^2_x} 
				&\leq 2R, \label{eq:bounded} \\
			\| \Phi (a, u) - \Phi(b, v) \|_{L^\infty_t H^1_x \, \cap \, L^2_t \dot H^2_x} 
				&\leq \frac12 \| (a, u) - (b, v) \|_{L^\infty_t H^1_x \, \cap \, L^2_t \dot H^2_x}, \label{eq:contract}
		\end{align}
	whenever 
		\begin{equation}\label{eq:ball}
			\| (a, u)  \|_{L^\infty_t H^1_x \, \cap \, L^2_t \dot H^2_x}, \| (b, v)  \|_{L^\infty_t H^1_x \, \cap \, L^2_t \dot H^2_x} \leq 2R. 
		\end{equation}
	By the Banach fixed point lemma, such a fixed point exists, is unique in $L^\infty_t H^1_x \cap L^2_t \dot H^2_x$, and thereby gives rise to the desired solution to \eqref{lambda-AHG}. 

	Towards the energy bound \eqref{bounded}, we apply the triangle inequality and the energy estimate \eqref{energy-estimate},
		\begin{equation}\label{eq:energy-step1}
		\begin{split}
			\| \Phi (a, u) \|_{L^\infty_t H^1_x \, \cap \, L^2_t \dot H^2_x} 
				&\leq \| e^{t \Delta} (a, u) \|_{L^\infty_t H^1_x \, \cap \, L^2_t \dot H^2_x}\\
					&\qquad  + \Bigl\| (\partial_t - \Delta)^{-1} \Bigl( \Delta(\underline A, \underline \phi) + \cN (A, \phi) \Bigr) \Bigr\|_{L^\infty_t H^1_x \, \cap \, L^2_t \dot H^2_x} \\
				&\leq R + C \Bigl( \| \Delta(\underline A, \underline \phi) \|_{L^2_{t, x}} + \| \cN (A, \phi) \|_{L^2_{t, x}} \Bigr).
		\end{split}
		\end{equation}
	for some constant $C > 0$. To handle the second term on the right of \eqref{energy-step1}, we write 
		\begin{align*}
			\| \Delta (\underline A, \underline \phi) \|_{L^2_{t, x}}
				&\lesssim T^\frac12 \Bigl( \| \nabla \underline F_{12} \|_{L^2_x} + \| \underline \bfD^{(2)} \underline \phi \|_{L^2_x} + \| \underline A\|_{L^\infty_x} \| \underline \bfD \underline \phi \|_{L^2_x} + \| \underline A \|_{L^4_x}^2 \| \underline \phi \|_{L^\infty_x} \Bigr) \lesssim T^\frac12,
		\end{align*}
	using the expression $\underline\bfD^j \underline\bfD_j \underline \phi = \Delta \underline \phi - 2i \underline A^j \underline \bfD_j \underline \phi + \underline A^j \underline A_j \underline \phi$ to estimate the Laplacian of the scalar field, the Coulomb gauge to handle the magnetic potential, and the properties of smooth configurations observed in Lemma \ref{lem:Linfty-config}. To handle the third term on the right of \eqref{energy-step1},
		\begin{align*}
			\| \cN (A, \phi) \|_{L^2_{t, x}} 
				&\lesssim \| (A, \phi) \|_{L^2_t L^\infty_x} \Bigl( \| \bfD \phi \|_{L^\infty_t L^2_x} + \| A \|_{L^\infty_t L^4_x}^2 + \| 1 - |\phi|^2\|_{L^\infty_t L^2_x} \Bigr) \\
				&\lesssim_R \Bigl( T^\frac12 + \| (a, u) \|_{L^2_t L^\infty_x} \Bigr) \Bigl( 1 + \| (a, u) \|_{L^\infty_t H^1_x} \Bigr) \\
				&\lesssim \Bigl( T^\frac12 + T^\frac14 \| (a, u) \|_{L^\infty_t L^2_x} \| (a, u) \|_{L^2_t \dot H^2_x} \Bigr) \Bigl( 1 + \| (a, u) \|_{L^\infty_t H^1_x} \Bigr) \lesssim_R T^\frac14
		\end{align*}
	using the triangle inequality, H\"older's inequality, the Lipchitz continuity of various non-linear expressions alotted by Proposition \ref{prop:lipschitz}, and Gagliardo-Nirenberg interpolation to handle the $L^\infty_x$-norm of $(a, u)$. Taking $T \ll 1$ in the previous two calculations and inserting back into \eqref{energy-step1} gives \eqref{bounded}.

	The contractive property \eqref{contract} follows a similar argument; using the energy estimate \eqref{energy-estimate}, suitable interpolations, and Proposition \ref{prop:lipschitz}, 
		\begin{align*}
			\| \Phi(a, u) - \Phi(b, v) \|_{L^\infty_t H^1_x \, \cap L^2_t \dot H^2_x} 
				&\leq  \Bigl\| (\partial_t - \Delta)^{-1} \Bigl(   \cN (\underline A + a, \underline \phi + u)  - \cN (\underline A + b, \underline \phi + v)  \Bigr) \Bigr\|_{L^\infty_t H^1_x \, \cap \, L^2_t \dot H^2_x} \\
				&\lesssim \Bigl\|   \cN (\underline A + a, \underline \phi + u)  - \cN (\underline A + b, \underline \phi + v)   \Bigr\|_{L^2_{t, x}}\\
				&\lesssim_R T^\frac14 \| (a, u) \|_{H^1}.
		\end{align*}
	Taking $T \ll 1$ allows us to conclude \eqref{contract} and thereby the proof. 
\end{proof}

\begin{proof}[Proof of energy identity]
	This is a standard calculation, multiplying the equation by $(A, \phi)$ and integrating-by-parts. The details may be carefully justified by considering smooth initial data $(a^{\mathrm{in}}, u^{\mathrm{in}}) \in H^\infty (\R^2)$; a similar argument to the one above will show persistence of this regularity. One can extend to $H^1$-regularity by density and Proposition \ref{prop:lipschitz}. 
\end{proof}

\begin{proof}[Proof of global existence]
	We claim that the solution may be continued as long as 
		\begin{equation}\label{eq:continuation}
			\int_0^T \| \phi(t) \|_{L^\infty_x}^2 \, dt < \infty. 
		\end{equation}
	This would imply global existence; indeed, using the bound \eqref{inefficientLinfty}, and using the smoothing estimate\footnote{we proved this for the self-dual case, however the estimate is robust and can easily be shown for $\lambda \neq 1$.} \eqref{smoothing} on unit-time scales depending only on the abelian Higgs energy, 
		\begin{align*}
			\int_0^T \| \phi(t) \|_{L^\infty_x}^2 \, dt 
				&\lesssim  \int_0^T 1 + \| 1 - |\phi(t)|^2 \|_{L^2_x}^{\frac13} \| \bfD^{(2)} \phi (t) \|_{L^2_x}^{\frac13} \, dt \lesssim_\cE T. 
		\end{align*}

	The continuation criterion \eqref{continuation} will be a consequence of the trivial continuation criterion given by the local existence theory, that is, (absence of) blow-up for the $H^1$-norm of $(a, u)$, and the energy estimate. Writing the solution to \eqref{lambda-AHG} as a Duhamel integral for the perturbation,  
		\[
			(a, u) (t) 
				=  e^{t \Delta} (a^{\mathrm{in}}, u^{\mathrm{in}}) + (\partial_t - \Delta)^{-1} \Bigl( \Delta(\underline A, \underline \phi) + \cN (A, \phi) \Bigr).
		\]
	then, Proposition \ref{lem:Linfty-config}, Gagliardo-Nirenberg, H\"older's inequality, and energy monotonicity, the perturbation of the scalar field obeys the energy estimate
		\begin{align*}
			\| u \|_{L^\infty_t H^1_x\, \cap \, L^2_t \dot H^2_x}
				&\lesssim \| u^{\mathrm{in}} \|_{H^1_x} + \| \Delta \underline \phi \|_{L^1_t L^2_x \, \cap \, L^2_{t, x}} + \| A \|_{L^1_t L^\infty_x \, \cap \, L^2_t L^\infty_x}\| \bfD \phi \|_{L^\infty_t L^2_x} \\
				&\qquad + \| \phi \|_{L^1_t L^\infty_x \, \cap \, L^2_t L^\infty_x} \Bigl( \| A \|_{L^\infty_t L^4_x}^2 + \| 1 - |\phi|^2 \|_{L^\infty_t L^2_x}\Bigr) \\
				&\lesssim_\cE 1 + T +  \| a \|_{L^\infty_t H^1_x\, \cap \, L^2_t \dot H^2_x} + \| \phi \|_{L^1_t L^\infty_x \, \cap \, L^2_t L^\infty_x} \bigl( 1 + \| a \|_{L^\infty_t H^1_x}^2 \bigr) .
		\end{align*}
	Thus, for $u$ to blow-up in $H^1$-norm, either $a$ blows-up in $L^\infty_t H^1_x \cap L^2_t \dot H^2_x$-norm, or \eqref{continuation} is violated. By the same circle of ideas, the perturbation of the magnetic field obeys the energy estimate, 
		\begin{align*}
			\| a \|_{L^\infty_t H^1_x\, \cap \, L^2_t \dot H^2_x}
				&\lesssim \| a^{\mathrm{in}} \|_{H^1_x} + \| \Delta \underline A\|_{L^1_t L^2_x \, \cap \, L^2_{t, x}} + \| \phi \|_{L^1_t L^\infty_x \, \cap \, L^2_t L^\infty_x} \| \bfD \phi \|_{L^\infty_t L^2_x} \\ 
				&\lesssim_\cE 1 + T + \int_0^T \| \phi \|_{L^\infty_x}^2 \, dt.
		\end{align*}
	We conclude the continuation criterion \eqref{continuation} and thereby the proof. 
\end{proof}

\bibliographystyle{alpha}
\bibliography{external/biblio.bib}
\end{document}